\documentclass[a4paper,twoside,reqno]{amsart}
\usepackage{amsthm,amsmath,amssymb,xcolor,soul,xspace,bm}
\usepackage[bookmarks=true,bookmarksopen=true]{hyperref}
\usepackage[ulem=normalem,draft]{changes}

\usepackage[ruled,algo2e]{algorithm2e}
\usepackage{graphicx,subcaption}

\theoremstyle{plain}
\newtheorem{theorem}{Theorem}[section]
\newtheorem{lemma}[theorem]{Lemma}
\newtheorem{corollary}[theorem]{Corollary}

\theoremstyle{definition}
\newtheorem{definition}[theorem]{Definition}

\newtheorem{assumption}[theorem]{Assumption}

\theoremstyle{remark}
\newtheorem{remark}[theorem]{Remark}

\numberwithin{equation}{section}
 \numberwithin{figure}{section}

\title[Optimal Control of Reaction-Diffusion Systems]{Optimal Control of a Class of Reaction-Diffusion Systems with Pointwise State Constraints}

\thanks{The first and third authors were supported by MICIU/AEI/10.13039/501100011033/ under research project PID2023-147610NB-I00.}

\author[E. Casas]{Eduardo Casas}
\address{Departamento de Matem\'{a}tica Aplicada y Ciencias de la Computaci\'{o}n, E.T.S.I. Industriales y de Telecomunicaci\'on, Universidad de Cantabria, 39005 Santander, Spain.}
\email{eduardo.casas@unican.es}

\author[K. Chrysafinos]{Konstantinos Chrysafinos}
\address{Department of Mathematics, School of Applied Mathematics and Physical Sciences, National Technical University of Athens, Zografou Campus, Athens 15780, Greece and IACM, FORTH, 20013 Heraklion, Crete, Greece.}
\email{chrysafinos@math.ntua.gr}

\author[M. Mateos]{Mariano Mateos}
\address{Departamento de Matem\'{a}ticas, Campus de Gij\'on, Universidad de Oviedo, 33203, Gij\'on, Spain.}
 \email{mmateos@uniovi.es}



\ifpdf
\hypersetup{
  pdftitle={Optimal Control of Reaction-Diffusion Systems},
  pdfauthor={E.~Casas,  K. Chrysafinos, and M.~Mateos}
}
\fi

\newcommand{\dx}{\,\mathrm{d}x}
\newcommand{\dt}{\,\mathrm{d}t}
\newcommand{\ds}{\,\mathrm{d}s}
\newcommand{\dmu}{\,\mathrm{d}\mu}
\newcommand{\dbmu}{\,\mathrm{d}\bar\mu}
\newcommand{\dabmu}{\,\mathrm{d}|\bar\mu|}
\newcommand{\supp}{\operatorname{supp}}
\newcommand{\Pb}{\mbox{\rm (P)}\xspace}
\newcommand{\proj}{\operatorname{Proj}}

\newcommand{\Y}{\mathcal{Y}}
\newcommand{\F}{\mathcal{F}}
\newcommand{\mG}{\mathcal{G}}
\newcommand{\mJ}{\mathcal{J}}
\newcommand{\mL}{\mathcal{L}}

\newcommand{\Uad}{U_{\rm ad}}
\newcommand{\K}{K_{\rm ad}}
\newcommand{\Aad}{A_{\rm ad}}
\newcommand{\M}{\mathcal{M}}
\newcommand{\mZ}{\mathcal{Z}}

\newcommand{\umin}{u_{\min}}
\newcommand{\umax}{u_{\max}}
\newcommand{\ymin}[1]{y_{\min,#1}}
\newcommand{\ymax}[1]{y_{\max,#1}}

\newcommand{\dive}{\operatorname{div}}

\begin{document}

\begin{abstract}
We study optimal control problems governed by a system of semilinear reaction-diffusion equations. The assumptions on the governing system are formulated so that our results apply to several models, including the Brusselator, Schnakenberg, and Gray--Scott systems. We prove the existence and uniqueness of a solution to the governing system without imposing any non-negativity assumption on the data. This allows us to study unconstrained, control-constrained, and state-constrained optimal control problems. For all these problems, we derive first-order necessary and second-order sufficient optimality conditions. The control acts on the activator equation through the boundary. Numerical examples demonstrate that both the inhibitor and the activator can be efficiently controlled. \end{abstract}

\maketitle

\section{Introduction} 
\label{S1}
In this paper, we analyze the optimal control problem:
\begin{align*}
&\Pb\quad \min_{u \in \Uad \text{ and } y_u \in \K} J(u),\\
J(u):=\int_Q L(x,t,y_u&(x,t))\dx\dt + \int_\Omega l(x,y_u(x,T))\dx + \frac{\nu}{2}\int_{\Sigma_c} u(x,t)^2\dx\dt,
\end{align*}
where $y_u = (y_{1,u},y_{2,u})$ denotes the solution of the following reaction-diffusion system
\begin{align}
&\left\{
\begin{array}{l}
 \displaystyle \frac{\partial y_1}{\partial t} - d_1\Delta y_1 
 + V_{1}\cdot\nabla y_1 + \gamma y_1y_2^2 + a_1y_1 + a_2y_2 = f_1\text{ in }Q, \\
  d_1\partial_n y_1 + \lambda_1y_1 = g_1\text{ on }\Sigma,\qquad y_1(0) = y_{0,1}\text{ in }\Omega, 
\end{array}
\right.
\label{E1.1} \\
&\left\{
\begin{array}{l}
\displaystyle  \frac{\partial y_2}{\partial t} - d_2\Delta y_2 
+ V_{2}\cdot \nabla y_2 - \gamma y_1 y^2_2 + b_1 y_1 + b_2 y_2 = f_2\text{ in }Q, \\
  d_2\partial_n y_2 
  +\lambda_2 y_2 = g_2 + u\chi_{\Sigma_c}\text{ on }\Sigma,\qquad y_2(0) = y_{0,2}\text{ in } \Omega. 
\end{array}
\right.
\label{E1.2} 
\end{align}
Above, $\nu > 0$, $\Omega \subset \mathbb{R}^n$ is a bounded open set with a Lipschitz boundary $\Gamma$ with $n = 2$ or 3, $\Gamma_c$ is a Lebesgue measurable subset of $\Gamma$ with positive measure, $T \in (0,\infty)$, $Q = \Omega \times (0,T)$, $\Sigma = \Gamma \times (0,T)$, and $\Sigma_c = \Gamma_c \times (0,T)$. With $\chi_{\Sigma_c}$ we denote the characteristic function of $\Sigma_c$. Regarding the control and state constraints, we assume that
\begin{align*}
&\Uad = \{u \in L^\infty(\Sigma_c) : \umin \le u(x,t) \le \umax \text{ for a.a. } (x,t) \in \Sigma_c\},\\
&\K = \{y \in C(\bar Q)^2 : \ymin{i} \le y_i(x,t) \le \ymax{i} \ \forall (x,t) \in \bar Q\text{ and } i=1,2\},\\
&\Aad = \{u \in \Uad : y_u \in \K\},
\end{align*}
where $-\infty < \umin < \umax < +\infty$ and $-\infty \leq \ymin{i} < \ymax{i} \leq +\infty$, $i = 1, 2$. Along this paper we assume that $\Aad$ is not empty.

The class of systems considered in this paper is important because of its many applications. For instance, the Brusselator model \cite{MichelPierre2010,PrigogineLefever1968}, obtained by taking
$V_1=V_2=\mathbf{0}$, $\gamma=1$, $a_1=0$, $a_2<0$, $b_1=0$, and
$b_2=-a_2+1$; the Schnakenberg model \cite{Madzvamuse2009,Schnakenberg1979},
with $b_2>0$, $a_1=a_2=b_1=0$, and $V_1=V_2$; and the Gray--Scott model
\cite{GrayScott1985,Pearson1993}, obtained by setting $V_1=V_2=\mathbf{0}$,
$a_1>0$, $a_2=0$, $b_1=0$, and $b_2>0$, all fit within the formulation
\eqref{E1.1}--\eqref{E1.2}. To the best of our knowledge, existence and uniqueness
for these systems have been established under the assumption that the initial
conditions and forcing terms are non-negative, which is the usual situation in
applications. In this case, the existence and uniqueness of non-negative
solutions are proved; see \cite{MichelPierre2010} and the references therein.
In this paper, we extend the analysis to forcing terms and initial conditions
without sign restrictions and prove the existence and uniqueness of a solution
in the space $W(0,T)\cap L^\infty(Q)$. As far as we know, this is the first
general result for this class of reaction-diffusion systems. The analysis of the
system is carried out in Section \ref{S2}.

In Section \ref{S3}, we formulate an optimal control problem associated with the
system \eqref{E1.1}--\eqref{E1.2}. The control acts from the boundary on the
equation satisfied by the activator $y_2$, and we impose pointwise constraints
on the activator and inhibitor $y_1$. We prove the existence of a solution to
the control problem and derive the first-order optimality conditions satisfied
by a local solution. The Lagrange multipliers associated with the state
constraints are Borel measures that appear in the adjoint state equation. The
existence, uniqueness, and regularity of the solution to the adjoint state
equation are established.

In Section \ref{S4}, the second-order analysis of the control problem is
performed. We provide a second-order sufficient optimality condition for a
strict local minimizer. To this end, an extended cone of critical directions
is introduced. Although many papers deal with the first-order analysis of
pointwise state-constrained control problems with distributed or Neumann
controls, very few provide second-order sufficient optimality conditions, and
almost all of them deal with distributed controls
\cite{Casas-Mateos-Rosch2025,RMRT2008,Krumbiegel-Rehberg-2013}. As far as we are aware, there is only one paper dealing with Neumann controls for an
elliptic equation in spatial dimension $2$, namely \cite{CRT2008}. Our paper is
the first to consider pointwise state-constrained control problems with Neumann
boundary controls in spatial dimension $2$ or $3$ governed by a semilinear
parabolic equation. We take advantage of the special structure of the system to
carry out the second-order analysis. This problem remains open for general semilinear parabolic partial differential equations.

Finally, in Section \ref{S5}, we present numerical results showing that the
system can be efficiently controlled from the boundary by acting on the
activator equation.

\section{Analysis of the state equation \eqref{E1.1}--\eqref{E1.2}}
\label{S2}
\setcounter{equation}{0}

In this section, we establish the existence and uniqueness of a solution to the system \eqref{E1.1}--\eqref{E1.2} and prove the differentiability of the relation control-to-state for suitable spaces. To do this, we make the following assumption.

\begin{assumption}
For $i = 1, 2$, we assume that $d_i > 0$ and $\lambda_i \ge 0$, and $\gamma > 0$ are constants, $a_i, b_i \in L^\infty(Q)$, $V_i \in L^\infty(Q)^n$ with $\dive{V_i} \in L^\infty(Q)$ and $V_i(x,t) \cdot \mathbf{n}(x) = 0$ on $\Sigma$, where $\mathbf{n}(x)$ denotes the outward unit normal vector to $\Gamma$ at the point $x$. We also assume that $y_0 = (y_{0,1},y_{0,2}) \in C(\bar\Omega)^2$, $f_i \in L^r(0,T;L^p(\Omega))$, $u \in L^s(0,T;L^q(\Gamma_c))$, $g_i \in L^s(0,T;L^q(\Gamma))$, $i = 1, 2$, with $r, p, s , q \ge 2$, $\frac{1}{r} + \frac{n}{2p} < 1$, and $\frac{1}{s} + \frac{n - 1}{2q} < \frac{1}{2}$.
\label{A2.1}
\end{assumption}

The assumptions on $r, p, s, q$ are necessary to have a continuous solution of the system \eqref{E1.1}--\eqref{E1.2}. Hereafter, we denote $a = (a_1,a_2)$, $b = (b_1,b_2)$, $f=(f_1,f_2)$ and $g = (g_1,g_2)$. As usual we set
\[
W(0,T) = \{w \in L^2(0,T;H^1(\Omega)) : \frac{\partial w}{\partial t} \in L^2(0,T;H^1(\Omega)^*)\}.
\]
It is well known that this vector space endowed with the norm
\[
\|w\|_{W(0,T)} = \Big(\|w\|^2_{L^2(0,T;H^1(\Omega))} + \|w\|^2_{L^2(0,T;H^1(\Omega)^*)}\Big)^{\frac{1}{2}}
\]
is a Hilbert space. The embeddings $W(0,T) \subset C([0,T];L^2(\Omega))$ and $W(0,T) \subset L^2(Q)$ are continuous, and the second embedding is even compact. In the sequel we denote
\[
\|w\|_{Q_t} = \|w\|_{L^\infty(0,t;L^2(\Omega))} + \|w\|_{L^2(0,t;H^1(\Omega))} \quad \forall t \in (0,T].
\]
If $t = T$ we simply write $\|w\|_Q$. We note that the norm $\|\cdot\|_Q$ is equivalent to
\[
\|w\| = \|w\|_{L^\infty(0,T;L^2(\Omega))} + \|\nabla w\|_{L^2(Q)}.
\]
Given a normed space $X$, if $w = (w_1,w_2) \in X^2$, we denote $\|w\|_X = \|w_1\|_X + \|w_2\|_X$.

\begin{definition}
We say that $y = (y_1,y_2) \in [W(0,T) \cap L^\infty(Q)]^2$ is a solution of \eqref{E1.1}--\eqref{E1.2} if the identities
\begin{align*}
\langle\frac{\partial y_1}{\partial t},w_1\rangle &+ \int_{\Omega}\big(d_1\nabla y_1 \nabla w_1 + [V_1\cdot \nabla y_1 + \gamma y_1y^2_2 +a\cdot y]w_1 \big)\dx + \lambda_1\int_\Gamma y_1w_1\dx\\ &= \int_{\Omega} f_1w_1\dx + \int_\Gamma g_1 w_1\dx\ \text{ in } \mathcal{D}'(0,T),\\
\langle\frac{\partial y_2}{\partial t},w_2\rangle &+ \int_{\Omega}\big(d_2\nabla y_2 \nabla w_2 + [V_2\cdot \nabla y_2 -  \gamma y_1y^2_2 +b\cdot y]w_2 \big)\dx+ \lambda_2\int_\Gamma y_2w_2\dx \\
&= \int_{\Omega} f_2w_2\dx + \int_\Gamma(g_2 + u\chi_{\Gamma_c})w_2\dx\ \text{ in } \mathcal{D}'(0,T),\\
&\hspace{-1.25cm}y_1(0) = y_{0,1},\, y_2(0) = y_{0,2},
\end{align*}
hold for all $w_i \in H^1(\Omega)$, $i = 1, 2$. Above, $\langle\cdot,\cdot\rangle$ denotes the duality pair between $H^1(\Omega)^*$ and $H^1(\Omega)$.
\label{D2.1}
\end{definition}

\begin{theorem}
Under the assumption \ref{A2.1}, the system \eqref{E1.1}--\eqref{E1.2} has a unique solution $y_u = (y_{u,1},y_{u,2}) \in [W(0,T) \cap C(\bar Q)]^2$. Moreover, there exist a constant $K$ and a monotone non-decreasing function $\eta:[0,\infty) \longrightarrow [0,\infty)$ independent of $f, u, g$, and $y_0$ such that
\begin{align}
&\|y_u\|_Q \le K\Big(\|f\|_{L^2(Q)} + \|g + (0,u)\chi_{\Sigma_c}\|_{L^2(\Sigma)} + \|y_0\|_{L^2(\Omega)} + \sqrt{T|\Omega|}\Big),\label{E2.1}\\
&\|y_u\|_{W(0,T)} + \|y_u\|_{C(\bar Q)}\label{E2.2}\\
&\phantom{\|y_u\|_Q}\le \eta\Big(\|f\|_{L^r(0,T;L^p(\Omega))} + \|g + (0,u)\chi_{\Gamma_c}\|_{L^s(0,T;L^q(\Gamma))} + \|y_0\|_{C(\bar\Omega)}\Big).\notag
\end{align}
Moreover, if $u_k \rightharpoonup u$ in $L^s(0,T;L^q(\Sigma_c))$, then $y_{u_k} \to y_u$ in $C(\bar Q)^2$ holds.
\label{T2.1}
\end{theorem}
\begin{proof}

{\em Existence of a solution.} We use the Schauder fixed-point theorem to prove the existence of a solution. Given $M > 0$ to be fixed later, let us define the function $F:L^2(Q) \longrightarrow L^2(Q)$ by $F(\phi) = y_{2,\phi}$, where $y_\phi = (y_{1,\phi},y_{2,\phi}) \in W(0,T)^2$ is the solution of the linear system
\begin{align}
&\left\{\begin{array}{l}\displaystyle \frac{\partial y_{1,\phi}}{\partial t} - d_1\Delta y_{1,\phi} + V_1\cdot\nabla y_{1,\phi} + \gamma y_{1,\phi}[\proj_{[-M,+M]}(\phi)]^2\\
\qquad \qquad + a_1y_{1,\phi} + a_2\proj_{[-M,+M]}(\phi) = f_1\ \text{ in } Q,\\d_1\partial_ny_{1,\phi} + \lambda_1y_{1,\phi} = g_1\ \text{ on } \Sigma,\ y_{1,\phi}(0) = y_{1,0} \ \text{ in } \Omega,\end{array}\right. \label{F2.3}\\
&\left\{\begin{array}{l}\displaystyle \frac{\partial y_{2,\phi}}{\partial t} - d_2\Delta y_{2,\phi} + V_2\cdot\nabla y_{2,\phi} - \gamma y_{1,\phi}[\proj_{[-M,+M]}(\phi)]^2\\
\qquad \qquad + b_1y_{1,\phi} + (b_2 + a_2)y_{2,\phi} - a_2\proj_{[-M,+M]}(\phi) = f_2\ \text{ in } Q,\\d_2\partial_ny_{2,\phi} + \lambda_2y_{2,\phi} = g_2\ \text{ on } \Sigma,\ y_{2,\phi}(0) = y_{2,0} \ \text{ in } \Omega.\end{array}\right. \label{F2.4}
\end{align}
Since the equations are decoupled, the existence and uniqueness of a solution of the above system are standard. In addition, there exists a constant $C_M$ depending on $M$ and the data of the equations such that $\|y_\phi\|_{W(0,T)^2} \le C_M$ for every $\phi \in L^2(Q)$.

The proof of existence of a solution of \eqref{E2.1}--\eqref{E2.2} is divided into four steps.\vspace{2mm}

{\em Step I - $\exists C_1 > 0$ independent of $M$ such that}
\begin{equation}
\|y_{1,\phi}\|_Q \le C_1(1 + \|f_1\|_{L^2(Q)} + \|g_1\|_{L^2(\Sigma)} + \|y_{1,0}\|_{L^2(\Omega)})\quad \forall \phi \in L^2(Q).
\label{F2.5}
\end{equation}
Setting $\alpha = \frac{1}{2}(d_1 + \frac{1}{d_1}\|V_1\|^2_{L^\infty(\Omega)^n} + 2\|a_1\|_{L^\infty(\Omega)})$, testing the equation \eqref{F2.3} with $\text{e}^{-2\alpha t}y_{1,\phi}$ and using that
\begin{align*}
&\gamma\int_\Omega y_{1,\phi}^2[\proj_{[-M,+M]}(\phi)]^2\dx + \int_\Omega a_2y_{1,\phi}\proj_{[-M,+M]}(\phi)\dx\\
&\qquad\qquad\ge \frac{\gamma}{2}\int_\Omega y_{1,\phi}^2[\proj_{[-M,+M]}(\phi)]^2\dx - \frac{1}{2\gamma}\|a_2\|^2_{L^\infty(\Omega)}|\Omega|,\\
& \int_{\Omega} V_1\cdot\nabla y_{1,\phi}y_{1,\phi}\dx \ge -\frac{d_1}{2}\|\nabla y_{1,\phi}\|^2_{L^2(\Omega)} - \frac{1}{2d_1}\|V_1\|^2_{L^\infty(\Omega)} \|y_{1,\phi}\|^2_{L^2(\Omega)},\\
&\exists C_\Gamma > 0 \text{ such that } \|w\|_{L^2(\Gamma)} \le C_\Gamma\|w\|_{H^1(\Omega)}\ \forall w \in H^1(\Omega),\\
&\langle\frac{\partial y_{1,\phi}}{\partial t},\text{e}^{-2\alpha t}y_{1,\phi}\rangle = \text{e}^{-2\alpha t}\langle\frac{\partial y_{1,\phi}}{\partial t},y_{1,\phi}\rangle = \text{e}^{-2\alpha t}\frac{1}{2}\frac{d}{dt}\|y_{1,\phi}(t)\|^2_{L^2(\Omega)}\\
&\qquad\qquad = \frac{1}{2}\frac{d}{dt}\Big(\text{e}^{-2\alpha t}\|y_{1,\phi}(t)\|^2_{L^2(\Omega)}\Big) + \alpha \text{e}^{-2\alpha t}\|y_{1,\phi}(t)\|^2_{L^2(\Omega)},
\end{align*}
we get
\begin{align*}
&\frac{1}{2}\frac{d}{dt}\Big(\text{e}^{-2\alpha t}\|y_{1,\phi}(t)\|^2_{L^2(\Omega)}\Big) + \text{e}^{-2\alpha t}\Big(\frac{d_1}{2}\|y_{1,\phi}(t)\|^2_{H^1(\Omega)} + \lambda_1\|y_{1,\phi}\|^2_{L^2(\Gamma)}\Big)\\
&\le \text{e}^{-2\alpha t}\Big(\frac{1}{2\gamma}\|a_2\|^2_{L^\infty(\Omega)}|\Omega| + \|f_1(t)\|_{L^2(\Omega)}\|y_{1,\phi}(t)\|_{L^2(\Omega)} + \|g_1(t)\|_{L^2(\Gamma)}\|y_{1,\phi}(t)\|_{L^2(\Gamma)}\Big).
\end{align*}
This implies
\begin{align*}
&\frac{d}{dt}\Big(\text{e}^{-2\alpha t}\|y_{1,\phi}(t)\|^2_{L^2(\Omega)}\Big) + \frac{d_1}{2}\text{e}^{-2\alpha t}\|y_{1,\phi}(t)\|^2_{H^1(\Omega)} \\
& \le \text{e}^{-2\alpha t}\Big(\frac{1}{\gamma}\|a_2\|^2_{L^\infty(\Omega)}|\Omega| + \frac{2}{d_1}\|f_1(t)\|^2_{L^2(\Omega)} + \frac{2}{d_1}C_\Gamma^2\|g_1(t)\|^2_{L^2(\Gamma)}\Big).
\end{align*}
Integrating from $0$ to $t$ we find
\begin{align*}
&\|y_{1,\phi}(t)\|^2_{L^2(\Omega)} + \frac{d_1}{2}\int_0^t\|y_{1,\phi}(s)\|_{H^1(\Omega)}^2\ds\\
& \le \text{e}^{2\alpha T}\Big(\frac{1}{\gamma}\|a_2\|^2_{L^\infty(\Omega)}|\Omega|T + \frac{2}{d_1}\|f_1\|^2_{L^2(Q)} + \frac{2}{d_1}C_\Gamma^2\|g_1\|^2_{L^2(\Sigma)} + \|y_{1,0}\|^2_{L^2(\Omega)}\Big).
\end{align*}
This proves \eqref{F2.5}, where the constant $C_1$ depends only on $d_1$, $\|V_1\|_{L^\infty(\Omega)^n}$, $\|a_i\|_{L^\infty(\Omega)}$, $i = 1, 2$, $\gamma$, $\Omega$, and $T$.\vspace{2mm}

{\em Step II - $\exists C_2 > 0$ independent of $M$ such that for every $\phi \in L^2(Q)$ the following inequality holds}
\begin{equation}
\|y_{2,\phi}\|_Q \le C_2\Big(1 + \|f\|_{L^2(Q)^2} + \|g\|_{L^2(\Sigma)^2} + \|y_0\|_{L^2(\Omega)^2}\Big). \label{F2.6}
\end{equation}

Setting $w_\phi = y_{1,\phi} + y_{2,\phi}$ and adding the corresponding equations we obtain
\begin{equation}
\left\{\begin{array}{l}\displaystyle \frac{\partial w_\phi}{\partial t} - d_2\Delta w_\phi + V_2\cdot \nabla w_\phi + (b_2 + a_2)w_\phi = f_1 + f_2\\
\qquad\qquad + (V_2 - V_1)\cdot\nabla y_{1,\phi} + (b_2 + a_2 - a_1 - b_1) y_{1,\phi} + (d_1 - d_2)\Delta y_{1,\phi}\ \text{ in } Q,\\d_2\partial_nw_\phi + \lambda_2w_\phi = g_1 + g_2 + (d_2 - d_1)\partial_ny_{1,\phi} +  (\lambda_2 - \lambda_1)y_{1,\phi}\ \text{ on } \Sigma,\\w_\phi(0) = y_{1,0} + y_{2,0} \ \text{ in } \Omega.\end{array}\right.\label{F2.7}
\end{equation}
We note that
\[
(d_2 - d_1)\partial_ny_{1,\phi} = \frac{d_2 - d_1}{d_1}(g_1 - \lambda_1y_{1,\phi}).
\]
Then, using \cite[Chapter III]{Lad-Sol-Ura68} or \cite[\S III.2]{Showalter1997} we deduce from \eqref{F2.5} and \eqref{F2.7} that
\begin{equation}
\|w_\phi\|_Q \le C\Big(1 + \|f\|_{L^2(Q)} + \|g\|_{L^2(\Sigma)} + \|y_0\|_{L^2(\Omega)}\Big),
\label{F2.8}
\end{equation}
where $C$ depends on $d_i$, $\lambda_i$, $a_i$, $b_i$, $\Omega$, and $T$. From  \eqref{F2.5} and \eqref{F2.8}, using that $\|y_{2,\phi}\|_Q \le \|w_\phi\|_Q + \|y_{1,\phi}\|_Q$, \eqref{F2.6} follows.\vspace{2mm}

{\em Step III - $\exists y_{2,M} \in L^2(Q)$ such that $F(y_{2,M}) = y_{2,M}$.} Taking
\begin{equation}
M \ge C_2\Big(1 + \|f\|_{L^2(Q)} + \|g\|_{L^2(\Sigma)} + \|y_0\|_{L^2(\Omega)}\Big),
\label{F2.9}
\end{equation}
we infer from \eqref{F2.6} that $\|F(\phi)\|_{L^2(Q)} = \|y_{2,\phi}\|_{L^2(Q)} \le M$. Therefore, we have that $F(\phi) \in W(0,T) \cap B_M$ for every $\phi \in L^2(Q)$, where $B_M$ stands for the closed ball of $L^2(Q)$ of radius $M$ and center 0. Since the embedding $W(0,T) \subset L^2(Q)$ is compact, we find that the image of $F$ is contained in a compact subset of $L^2(Q)$. Hence, we can apply Schauder fixed-point theorem to deduce the existence of a fixed point $y_{2,M}$ for $F$. Replacing $\phi$ by $y_{2,M}$ in \eqref{E2.7} and \eqref{E2.8} we obtain that $y_M = (y_{1,M},y_{2,M}) \in W(0,T)^2$ satisfies the equations
\begin{align}
&\left\{\begin{array}{l}\displaystyle \frac{\partial y_{1,M}}{\partial t} - d_1\Delta y_{1,M} + V_1\cdot\nabla y_{1,M} + \gamma y_{1,M}[\proj_{[-M,+M]}(y_{2,M})]^2\\
\qquad\qquad +a_1y_{1,M} + a_2\proj_{[-M,+M]}(y_{2,M}) = f_1\ \text{ in } Q,\\d_1\partial_ny_{1,M} + \lambda_1y_{1,M} = g_1\ \text{ on } \Sigma,\ y_{1,M}(0) = y_{1,0} \ \text{ in } \Omega,\end{array}\right. \label{F2.10}\\
&\left\{\begin{array}{l}\displaystyle \frac{\partial y_{2,M}}{\partial t} - d_2\Delta y_{2,M} + V_2\cdot\nabla y_{2,M} - \gamma y_{1,M}[\proj_{[-M,+M]}(y_{2,M})]^2\\
\qquad\qquad + b_1y_{1,M} + (b_2 + a_2)y_{2,M} - a_2\proj_{[-M,+M]}(y_{2,M}) = f_2\ \text{ in } Q,\\d_2\partial_ny_{2,M} + \lambda_2y_{2,M} = g_2\ \text{on } \Sigma,\ y_{2,M}(0) = y_{2,0} \ \text{ in } \Omega.\end{array}\right. \label{F2.11}
\end{align}

{\em Step IV - There exists $M > 0$ satisfying \eqref{F2.9} such that $y_M$ is a solution of \eqref{E1.1}--\eqref{E1.2} and satisfies \eqref{E2.1} and \eqref{E2.2}.}
We write the equation \eqref{F2.10} as
\[
\left\{\begin{array}{l}\displaystyle \frac{\partial y_{1,M}}{\partial t} - d_1\Delta y_{1,M} + V_1\cdot\nabla y_{1,M} + \Big(a_1 + \gamma [\proj_{[-M,+M]}(y_{2,M})]^2\Big)y_{1,M} = \tilde{f}\ \text{ in } Q,\\ d_1\partial_ny_{1,M} + \lambda_1y_{1,M} = g_1\ \text{ on } \Sigma,\ y_{1,M}(0) = y_{1,0} \ \text{ in } \Omega,\end{array}\right.
\]
where $\tilde{f} = f_1 - a_2\proj_{[-M,+M]}(y_{2,M})$.

Let us observe that $L^\infty(0,T;L^2(\Omega)) \cap L^2(0,T;H^1(\Omega)) \subset L^3(Q)$. Indeed, interpolating $L^3(\Omega)$ between $L^2(\Omega)$ and $L^6(\Omega)$ we infer
\[
\|w\|_{L^3(\Omega)} \le \|w\|_{L^2(\Omega)}^{\frac{1}{2}}\|w\|_{L^6(\Omega)}^{\frac{1}{2}} \le C_\Omega\|w\|_{L^2(\Omega)}^{\frac{1}{2}}\|w\|_{H^1(\Omega)}^{\frac{1}{2}}\quad \forall w \in H^1(\Omega).
\]
This yields
\begin{align}
\|w\|_{L^3(Q)} &\le T^{\frac{1}{12}}\|w\|_{L^4(0,T;L^3(\Omega))} \le T^{\frac{1}{12}}C_\Omega\|w\|_{L^\infty(0,T;L^2(\Omega))}^{\frac{1}{2}}\|w\|^{\frac{1}{2}}_{L^2(0,T;H^1(\Omega))}\notag\\& \le \frac{T^{\frac{1}{12}}C_\Omega}{2}\|w\|_Q \quad \forall w \in L^\infty(0,T;L^2(\Omega)) \cap L^2(0,T;H^1(\Omega)).\label{F2.12}
\end{align}

From \eqref{F2.5}, \eqref{F2.6}, and \eqref{F2.12} we deduce that
\begin{align*}
&\|y_{1,M}\|_{L^3(Q)} + \|\proj_{[-M,+M]}(y_{2,M})\|_{L^3(Q)} \le \|y_M\|_{L^3(Q)}\\
& \le C_3\Big(1 + \|f\|_{L^2(Q)} + \|g\|_{L^2(\Sigma)} + \|y_0\|_{L^2(\Omega)}\Big).
\end{align*}
Noting that $\frac{1}{3} + \frac{n}{6} < 1$, we infer from \cite[\S III.7]{Lad-Sol-Ura68}, see also \cite{CasasKunisch2025}, that $y_{1,M} \in L^\infty(Q)$ and
\begin{align}\label{F2.13}
\|y_{1,M}\|_{L^\infty(Q)}   \le C_4\Big( & 1 + \|f_1\|_{L^r(0,T;L^p(\Omega))} + \|g_1\|_{L^s(0,T;L^q(\Gamma))} + \|y_0\|_{L^\infty(\Omega)} \\
& + \|f\|_{L^2(Q)} + \|g\|_{L^2(\Sigma)} \Big), \notag
\end{align}
where $C_4$ is independent of $M$. Now, we prove the boundedness of $y_{2,M}$ by a constant independent of $M$. For this purpose, we first consider the equation
\[
\left\{\begin{array}{l}
\displaystyle \frac{\partial w}{\partial t} - d_2\Delta w + V_2\cdot\nabla w  + w = f_2 -  (b_2 + a_2)y_{2,M} - b_1y_{1,M}\\\phantom{\frac{\partial w}{\partial t} - d_2\Delta w + V_2\cdot\nabla w  + w =} + a_2\proj_{[-M,+M]}(y_{2,M})\,\text{in}\, Q,\\
d_2\partial_nw + \lambda_2w = g_2\ \text{ on } \Sigma,\ w(0) = y_{2,0} \ \text{ in } \Omega.\end{array}\right.
\]
This equation has a unique solution $w \in W(0,T)$. Moreover, using again \cite[\S III.7]{Lad-Sol-Ura68} or \cite{CasasKunisch2025} we get the existence of a constant $C_5$ independent of $M$ such that
\begin{align*}
\|w\|_{L^\infty(Q)}   
\le C_5\Big( & \|f_2\|_{L^r(0,T;L^p(\Omega))} + [\|a_2\|_{L^\infty(\Omega)}+\|b_2\|_{L^\infty(\Omega)}] \|y_{2,M}\|_{L^3(Q)} \\
& + \|b_1\|_{L^\infty(\Omega)}\|y_{1,M}\|_{L^3(Q)}
+ \|g_2\|_{L^s(0,T;L^q(\Gamma))} + \|y_{2,0}\|_{L^\infty(\Omega)}\Big).
\end{align*}
Combining this inequality with \eqref{F2.5} and \eqref{F2.6} we obtain
\begin{align*}
\|w\|_{L^\infty(Q)} &\le C_5\Big(\|f_2\|_{L^r(0,T;L^p(\Omega))}  + \|g_2\|_{L^s(0,T;L^q(\Gamma))} + \|y_{2,0}\|_{L^\infty(\Omega)}\\
&+ C_6\Big[1 + \|f\|_{L^2(Q)} + \|g\|_{L^2(\Sigma)} + \|y_0\|_{L^2(\Omega)}\Big]\Big).
\end{align*}
Setting $z = y_{2,M} - w$ we have
\[
\left\{\begin{array}{l}\displaystyle \frac{\partial z}{\partial t} - d_2\Delta z + V_2\cdot\nabla z + (a+1)z = -aw + y_{2,M}\ \text{ in } Q,\\ d_2\partial_nz + \lambda_2z = 0\ \text{ on } \Sigma,\ z(0) = 0 \ \text{ in } \Omega,\end{array}\right.
\]
where
\[
a(x,t) = \left\{\begin{array}{cl}\displaystyle -\gamma y_{1,M}(x,t)\frac{[\proj_{[-M,+M]}(y_{2,M}(x,t))]^2}{y_{2,M}(x,t)}&\text{if } y_{2,M}(x,t) \neq 0,\\0&\text{otherwise}.\end{array}\right.
\]
From \eqref{F2.13} and \eqref{F2.6} we get that $\|a\|_{L^3(Q)} \le \gamma\|y_{1,M}\|_{L^\infty(Q)}\|y_{2,M}\|_{L^3(Q)} \le C_7$ for some constant $C_7$ independent of $M$. Now, from \cite[Theorem III-7.1]{Lad-Sol-Ura68} we obtain that $\|z\|_{L^\infty(Q)} \le C_8$ for some constant depending on $\|a\|_{L^3(Q)}$ and $\|w\|_{L^\infty(Q)}$. Taking into account that $y_{2,M} = w + z$, we conclude that $y_{2,M}$ is bounded by a constant $C_9$ independent of $M$. Therefore, if we select $M > C_9$ satisfying \eqref{F2.9} we get that $\proj_{[-M,+M]}(y_{2,M}) = y_{2,M}$ and consequently $y_M = (y_{1,M},y_{2,M}) \in [W(0,T) \cap L^\infty(Q)]^2$ is a solution of \eqref{E1.1}--\eqref{E1.2} and satisfies \eqref{E2.1} and \eqref{E2.2}.\vspace{2mm}

{\em Uniqueness.} Let us assume that $y, z \in [W(0,T) \cap L^\infty(Q)]^2$ are solutions of the system \eqref{E1.1}--\eqref{E1.2}. Setting $w = y - z$ and subtracting the corresponding equations we get
\begin{align}
&\left\{\begin{array}{l}\displaystyle \frac{\partial w_1}{\partial t} - d_1\Delta w_1 + V_1\cdot\nabla w_1 + \gamma w_1y_2^2 + a_1w_1\\
\qquad+ [a_2 + \gamma z_1(y_2+z_2)]w_2 = 0\ \text{ in } Q,\\d_1\partial_nw_1 + \lambda_1w_1 = 0\ \text{ on } \Sigma,\ w_1(0) = 0 \ \text{ in } \Omega,\end{array}\right. \label{F2.14}\\
&\left\{\begin{array}{l}\displaystyle \frac{\partial w_2}{\partial t} - d_2\Delta w_2 + V_2\cdot\nabla w_2 + [b_2 -\gamma z_1(y_2 + z_2)]w_2\\
\qquad+ (b_1 - \gamma y_2^2)w_1 = 0 \ \text{ in } Q,\\d_2\partial_nw_2 + \lambda_2w_2 = 0\ \text{ on } \Sigma,\ w_2(0) = 0 \ \text{ in } \Omega,\end{array}\right. \label{F2.15}
\end{align}
Testing \eqref{F2.14} with $w_1$ and \eqref{F2.15} with $w_2$, integrating in $\Omega \times (0,t)$, adding the resulting identities, taking into account that the coefficients belong to $L^\infty(Q)$, and using Gronwall's inequality $w = 0$ follows.
\end{proof}
The estimate \eqref{E2.2} can be improved if we only look at $\|y_{u,1}\|_{C(\bar Q)}$. In this case, we get the inequality replacing $\|g_2 + u\chi_{\Sigma_c}\|_{L^s(0,T;L^q(\Gamma))}$ by $\|g_2 + u\chi_{\Sigma_c}\|_{L^2(\Sigma)}$; see \eqref{F2.13}. We will use the following result.
\begin{lemma}
Suppose that the assumption \ref{A2.1} holds. Then, there exist constants $K_\infty$, $K_2$, and $K_\sigma$ with $\sigma > n + 1$  such that
\begin{equation}
\left\{\begin{array}{l}\displaystyle\|y_{u_2,1} - y_{u_1,1}\|_{C(\bar Q)} \le K_\infty\|u_2 - u_1\|_{L^2(\Sigma_c)},\\\displaystyle\|y_{u_2} - y_{u_1}\|_{W(0,T)} \le K_2\|u_2 - u_1\|_{L^2(\Sigma_c)},\\\displaystyle\|y_{u_2} - y_{u_1}\|_{C(\bar Q)} \le K_\sigma\|u_2 - u_1\|_{L^\sigma(\Sigma_c)},
\end{array}\right.\quad \forall u_1, u_2 \in \Uad.
\label{E2.3}
\end{equation}
\label{L2.1}
\end{lemma}

\begin{proof}
Set $e = y_{u_2} - y_{u_1}$. Subtracting the equations satisfied by $y_{u_2}$ and $y_{u_1}$ we obtain
\begin{align}
&\left\{
\begin{array}{l}
 \displaystyle\frac{\partial e_1}{\partial t} - d_1\Delta e_1 
 + V_{1}\cdot\nabla e_1 + (a_1 + \gamma y_{u_2,2}^2)e_1\\ \phantom{\frac{\partial e_1}{\partial t} - d_1\Delta e_1 } + [\gamma y_{u_1,1}(y_{u_2,2} + y_{u_1,2}) + a_2]e_2 = 0\text{ in }Q, \\
d_1\partial_n e_1 + \lambda_1e_1 = 0\text{ on }\Sigma,\qquad e_1(0) = 0\text{ in }\Omega, 
\end{array}
\right.
\label{E2.4} \\
&\left\{
\begin{array}{l}
\displaystyle  \frac{\partial e_2}{\partial t} - d_2\Delta e_2 
+ V_{2}\cdot\nabla e_2 + [b_2 - \gamma y_{u_1,1}(y_{u_2,2} + y_{u_1,2})]e_2\\ \phantom{\frac{\partial e_1}{\partial t} - d_1\Delta e_1 }  + (b_1 - \gamma y^2_{u_2,2})e_1 = 0\text{ in }Q, \\
  d_2\partial_n e_2 + \lambda_2e_2 = (u_2 - u_1)\chi_{\Sigma_c}\text{ on }\Sigma,\qquad e_2(0) = 0\text{ in }\Omega. 
\end{array}
\right.
\label{E2.5} 
\end{align}
From \eqref{E2.2} we infer the existence of a constant $C$ such that $\|y_u\|_{C(\bar Q)} \le C$ for every $u \in \Uad$. Then, all the coefficients in the linear system above are bounded functions. As a consequence the existence and uniqueness of a solution $e \in W(0,T)^2$ follows; see \cite[Chapter VII]{Lad-Sol-Ura68} or \cite[\S III.2]{Showalter1997}. We also have that $\|e\|_{W(0,T)} \le K_2\|u_2 - u_1\|_{L^2(\Sigma_c)}$ for every $u_1, u_2 \in \Uad$. Using this estimate and again \cite[Chapter III-\S 7]{Lad-Sol-Ura68} and \cite[Chapter III-\S 10]{Lad-Sol-Ura68} we deduce from \eqref{E2.4} that
\[
\|e_1\|_{C(\bar Q)} \le C\|[\gamma y_{u_1,1}(y_{u_2,2} + y_{u_1,2}) + a_2]e_2\|_{L^\infty(0,T;L^2(\Omega))} \le K_\infty\|u_2 - u_1\|_{L^2(\Sigma_c)}.
\]
This proves \eqref{E2.3}. The last estimate of the theorem follows from the above estimate and the results of \cite{CasasKunisch2025} applied to the equation \eqref{E2.5}.
\end{proof}

From Theorem \ref{T2.1} we know that $G:L^s(0,T;L^q(\Gamma_c)) \longrightarrow [W(0,T) \cap C(\bar Q)]^2$ is a well defined mapping. Next we establish its differentiability.

\begin{theorem}
The mapping $G$ is of class $C^\infty$. Moreover, given $u, v, v_1, v_2 \in L^s(0,T;L^q(\Gamma_c))$, we have that $z_{u,v} = DG(u)v$ and $z_{u,v_1,v_2} = D^2G(u)(v_1,v_2)$ are the unique solutions of the equations
\begin{align}
&\left\{\begin{array}{l}
 \displaystyle \frac{\partial z_1}{\partial t} - d_1\Delta z_1 + V_{1}\cdot\nabla z_1 + (\gamma y_{u,2}^2 + a_1)z_1 + (a_2 + 2\gamma y_{u,1}y_{u,2})z_2 = 0\text{ in }Q, \\
  d_1\partial_n z_1 + \lambda_1z_1 = 0\text{ on }\Sigma,\qquad z_1(0) = 0\text{ in }\Omega, 
\end{array}
\right.
\label{E2.6} \\
&\left\{
\begin{array}{l}
\displaystyle  \frac{\partial z_2}{\partial t} - d_2\Delta z_2 
+ V_{2}\cdot\nabla z_2 + (b_2 - 2\gamma y_{u,1}y_{u,2})z_2 + (b_1 - \gamma y_{u,2}^2)z_1 = 0\text{ in }Q, \\ d_2\partial_n z_2 +\lambda_2 z_2 = v\chi_{\Sigma_c}\text{ on }\Sigma,\qquad z_2(0) = 0\text{ in }\Omega, 
\end{array}
\right.
\label{E2.7} 
\end{align}
and
\begin{align}
&\left\{\begin{array}{l}
 \displaystyle \frac{\partial z_1}{\partial t} - d_1\Delta z_1 + V_{1}\cdot\nabla z_1 + (\gamma y_{u,2}^2 + a_1)z_1 + (a_2 + 2\gamma y_{u,1}y_{u,2})z_2\\
 \phantom{\displaystyle \frac{\partial z_1}{\partial t} - d_1\Delta z_1 + V_{1}\cdot\nabla z_1} + 4\gamma y_{u,2}z_{u,v,1}z_{u,v,2} + 2\gamma y_{u,1}z_{u,v,2}^2 = 0\text{ in }Q, \\
  d_1\partial_n z_1 + \lambda_1z_1 = 0\text{ on }\Sigma,\qquad z_1(0) = 0\text{ in }\Omega, 
\end{array}
\right.
\label{E2.8} \\
&\left\{
\begin{array}{l}
\displaystyle  \frac{\partial z_2}{\partial t} - d_2\Delta z_2 
+ V_{2}\cdot\nabla z_2 + (b_2 - 2\gamma y_{u,1}y_{u,2})z_2 + (b_1 - \gamma y_{u,2}^2)z_1\\
\phantom{\displaystyle \frac{\partial z_1}{\partial t} - d_1\Delta z_1 + V_{1}\cdot\nabla z_1} - 4\gamma y_{u,2}z_{u,v,1}z_{u,v,2} - 2\gamma y_{u,1}z_{u,v,2}^2 = 0\text{ in }Q, \\ d_2\partial_n z_2 +\lambda_2 z_2 = 0\text{ on }\Sigma,\qquad z_2(0) = 0\text{ in }\Omega, 
\end{array}
\right.
\label{E2.9} 
\end{align}
respectively, where $z_{u,v} = (z_{u,v,1},z_{u,v,2})$.
\label{T2.2}
\end{theorem}
    
\begin{proof}
We define the vector space
\begin{align*}
\Y = \{y \in [W(0,T) \cap C(\bar Q)]^2 : &\frac{\partial y_i}{\partial t} - d_i\Delta y_i 
 + V_i\cdot\nabla y_i \in L^r(0,T;L^p(\Omega)) \text{ and }\\
 &\partial_ny_i \in L^s(0,T;L^q(\Gamma)), \ i = 1, 2\}.
\end{align*}
This is a Banach space when endowed with the graph norm. We define the mapping
\begin{align*}
&\F:\Y \times L^s(0,T;L^q(\Gamma_c)) \longrightarrow L^r(0,T;L^p(\Omega))^2 \times L^s(0,T;L^q(\Gamma))^2 \times C(\bar\Omega)^2,\\
&\F(y,u) = \left[\begin{array}{c}\displaystyle \frac{\partial y_1}{\partial t} - d_1\Delta y_1 + V_1\cdot\nabla y_1 + \gamma y_1y_2^2 + a\cdot y - f_1\vspace{2mm}\\\displaystyle\frac{\partial y_2}{\partial t} - d_2\Delta y_2 + V_2\cdot\nabla y_2 - \gamma y_1y_2^2 + b\cdot y - f_2\\\displaystyle d_1\partial_n y_1 + \lambda_1y_1 - g_1\\d_2\partial_ny_2 + \lambda_2y_2 - g_2 - u\chi_{\Sigma_c}\\y(0) - y_0\end{array}\right].
\end{align*}
Obviously $\F$ is of class $C^\infty$ and $\F(G(u),u) = 0$ for every $u \in L^s(0,T;L^q(\Gamma_c))$. Using the implicit function theorem we deduce that $G$ is of class $C^\infty$. To this end we only need to prove that the linear operator
\begin{align*}
&\frac{\partial\F}{\partial y}(y_u,u):\Y \longrightarrow L^r(0,T;L^p(\Omega))^2 \times L^s(0,T;L^q(\Gamma))^2 \times C(\bar\Omega)^2\\
&\frac{\partial\F}{\partial y}(y_u,u)z = \left[\begin{array}{c}\displaystyle \frac{\partial z_1}{\partial t} - d_1\Delta z_1 + V_1\cdot\nabla z_1 + (a_1 + \gamma y_{u,2}^2)z_1 + (a_2 + 2\gamma y_{u,1}y_{u,2})z_2\vspace{2mm}\\\displaystyle\frac{\partial z_2}{\partial t} - d_2\Delta z_2 + V_2\cdot\nabla z_2 + (b_2 - 2\gamma y_{u,1}y_{u,2})z_2 + (b_1 - \gamma y_{u,2}^2)z_1\\\displaystyle d_1\partial_n z_1 + \lambda_1z_1\\d_2\partial_nz_2 + \lambda_2z_2\\z(0)\end{array}\right]
\end{align*}
is an isomorphism. Using again \cite[Chapter VII]{Lad-Sol-Ura68} or \cite[\S III.2]{Showalter1997} we find that the equation $\frac{\partial\F}{\partial y}(y_u,u)z = (h,g,z_0)$ has a unique solution $z$ in $W(0,T)^2$ for every $h \in L^r(0,T;L^p(\Omega))^2$, $g \in L^s(0,T;L^q(\Gamma))^2$, and $z_0 \in C(\bar\Omega)^2$. Moreover, applying \cite[Chapter III-\S 7 and \S 10]{Lad-Sol-Ura68} we also get that $z \in C(\bar Q)^2$. Finally, the equations \eqref{E2.6}--\eqref{E2.9} follow by differentiation of the identity $\F(G(u),u) = 0$.
\end{proof}

We end this section with the following technical lemma.

\begin{lemma}
For every $v \in L^2(\Sigma_c)$ and $u \in L^s(0,T;L^q(\Gamma_c))$, the system \eqref{E2.6}--\eqref{E2.7} has a unique solution $z_{u,v}$ in $W(0,T)$. Further, we have that $z_{u,v,1} \in W(0,T) \cap C(\bar Q)$ and $z_{u,v,2} \in W(0,T) \cap L^{\sigma_0}(Q)$ for every $\sigma_0 < \frac{2(n + 2)}{n - 1}$, and there exist constants $K_W$, $K_1$ and $K_\sigma$ such that for every $u \in \Uad$
\begin{align}
&\|z_{u,v}\|_{W(0,T)} \ \, \le K_W\|v\|_{L^2(\Sigma_c)},\label{E2.10}\\
&\|z_{u,v,1}\|_{C(\bar Q)} \ \ \le K_1\|v\|_{L^2(\Sigma_c)},
\label{E2.11}
\\
&\|z_{u,v,2}\|_{L^{\sigma_0}(Q)} \le K_{\sigma_0}\|v\|_{L^2(\Sigma_c)}.
\label{E2.12}
\end{align}
\label{L2.2}
\end{lemma}
\begin{proof}
From \cite[Chapter VII]{Lad-Sol-Ura68} or \cite[\S III.2]{Showalter1997} we know the existence of a unique solution $z_{u,v} \in W(0,T)$ for every $u \in L^s(0,T;L^q(\Gamma_c))$ and $v \in L^2(\Sigma_c)$. Furthermore, there exists a constant $K_W$ such that $\|z_{u,v}\|_{W(0,T)} \le K_W\|v\|_{L^2(\Sigma_c)}$ for every $u \in \Uad$. Then, using this regularity of $z_{u,v,2}$ and applying \cite[Chapter III-\S 10]{Lad-Sol-Ura68} to the equation \eqref{E2.6} we deduce the estimate \eqref{E2.11}. To prove \eqref{E2.12} we proceed by interpolation. For every $1 \le \sigma_1 < \frac{n + 2}{n}$ the estimate $\|z_{u,v,2}\|_{L^{\sigma_1}(Q)} \le C_{\sigma_1}\|v\|_{L^1(\Sigma_c)}$ holds. Indeed, from \cite[Theorem 6.3]{Casas-1997} we know that $z_{u,v,2} \in L^{r_1}(0,T;W^{1,p_1}(\Omega))$ for every $r_1, p_1 \in [1,2)$ satisfying $\frac{2}{r_1} + \frac{n}{p_1} > n + 1$ and there exists a constant $C_{r_1,p_1}$ such that $\|z_{u,v,2}\|_{L^{r_1}(0,T;W^{1,p_1}(\Omega))} \le C_{r_1,p_1}\|v\|_{L^1(\Sigma_c)}$. Now, it suffices to observe that for every $\sigma_1 < \frac{n + 2}{n}$ there exist $r_1$ and $p_1$ as above such that $L^{r_1}(0,T;W^{1,p_1}(\Omega)) \subset L^{\sigma_1}(Q)$ to deduce the estimate. On the other hand, for every $q_1 > n + 1$ and $v \in L^{q_1}(\Sigma_c)$, we get  $\|z_{u,v,2}\|_{L^\infty(Q)} \le C_2\|v\|_{L^{q_1}(\Sigma_c)}$. Since $\frac{1}{2} = \frac{\alpha}{1} + \frac{1 - \alpha}{q_1}$ for $\alpha = \frac{q_1 - 2}{2(q_1 - 1)}$, we obtain by interpolation that $z_{u,v,2} \in L^\sigma(Q)$ for $\sigma = \frac{\sigma_1}{\alpha}$. Using that $\sigma_1 < \frac{n + 2}{n}$ and $q_1 > n +1$ can be taken arbitrarily close to $\frac{n + 2}{n}$ and $n + 1$, respectively, we deduce that $\sigma < \frac{2(n + 2)}{n -1}$, but it can also be arbitrarily close to $\frac{2(n + 2)}{n -1}$. The interpolation inequalities imply \eqref{E2.12}.
\end{proof}

\begin{remark}
As proved in the previous lemma, we have that $z_{u,v,1} \in C(\bar Q)$ for every $u \in \Uad$ and every $v \in L^2(\Sigma_c)$. On the contrary, as shown in theorem \ref{T2.2}, regularity $z_{u,v,2} \in C(\bar Q)$ requires the assumption $v \in L^s(0,T;L^q(\Gamma_c))$.
\label{R2.1}
\end{remark}

\section{Analysis of the control problem \Pb}
\label{S3}
\setcounter{equation}{0}

In this section, we establish the existence of a solution for the control \Pb and derive the first-order optimality conditions satisfied by a local minimizer. For this purpose we make the following assumption on the cost functional $J$.

\begin{assumption}
We assume that $L:Q \times \mathbb{R}^2 \longrightarrow \mathbb{R}$ and $l:\Omega \times \mathbb{R}^2 \longrightarrow \mathbb{R}$ are Carath\'eodory functions of class $C^2$ with respect to the last variable. Moreover, we suppose that the following conditions hold for almost all $(x,t) \in Q$:
\begin{align}
&L(\cdot,\cdot,0) \in L^1(Q)\ \text{ and }\ l(\cdot,0) \in L^1(\Omega),\notag\\
&\left\{\begin{array}{l}\forall M > 0 \ \exists \psi_M \in L^r(0,T;L^p(\Omega)) \text{ and } C_M > 0 \text{ such that for $|y| \le M$}\\
\displaystyle\Big|\frac{\partial L}{\partial y}(x,t,y)\Big| \le \psi_M(x,t) \text{ and } \sum_{j = 1}^2\Big|\frac{\partial^j l}{\partial y^j}(x,y)\Big| + \Big|\frac{\partial^2L}{\partial y^2}(x,t,y)\Big| \le C_M,\end{array}\right.\label{E3.1}\\
&\left\{\begin{array}{l}\forall\varepsilon > 0 \text{ and } \forall M > 0 \ \ \exists \rho_{\varepsilon,M} > 0 : \text{ if } |y_i| \le M \text{ and } |y_2 - y_1|< \rho_{\varepsilon,M}\\\displaystyle
\Big|\frac{\partial^2l}{\partial y^2}(x,y_2) - \frac{\partial^2l}{\partial y^2}(x,y_1)\Big| + \Big|\frac{\partial^2L}{\partial y^2}(x,t,y_2) - \frac{\partial^2L}{\partial y^2}(x,t,y_1)\Big| < \varepsilon.\end{array}\right.\label{E3.2}
\end{align}
We also assume that $\ymin{i} < y_{0,i}(x) < \ymax{i}$ for every $x \in \bar\Omega$ for $i = 1, 2$.
\label{A3.1}
\end{assumption}

Next we deduce the following differentiability property of $J$.

\begin{theorem}
The function $J{:}L^s(0,T;L^q(\Gamma_c)) \longrightarrow \mathbb{R}$ is of class $C^2$ and we have
\begin{align}
&J'(u)v = \int_{\Sigma_c}(\varphi_{u,2} + \nu u)v\dx\dt, \label{E3.3}\\
&J''(u)v^2 = \nu \int_{\Sigma_c}v^2\dx\dt + \int_\Omega\frac{\partial^2l}{\partial y^2}(x,y_u(T))z_{u,v}^2\dx\label{E3.4}\\
&+\int_Q\Big[\frac{\partial^2L}{\partial y^2}(x,t,y_u)z^2_{u,v} {+} 2\gamma(\varphi_{u,2} {-} \varphi_{u,1})(2y_{u,2}z_{u,v,1}z_{u,v,2} {+} y_{u,1}z_{u,v,2}^2)\Big]\dx\dt,\notag
\end{align}
where $z_{u,v} = DG(u)v$ and $\varphi_u \in [W(0,T) \cap L^\infty(Q)]^2$ is the solution of the system
\begin{align}
&\left\{\begin{array}{l}\displaystyle -\frac{\partial\varphi_1}{\partial t} - d_1\Delta\varphi_1 - \dive{(V_1\varphi_1)}\\\phantom{-\frac{\partial\varphi_1}{\partial t}}\displaystyle + (\gamma y_{u,2}^2 + a_1)\varphi_1 + (b_1 - \gamma y_{u,2}^2)\varphi_2 = \frac{\partial L}{\partial y_1}(x,t,y_u)\ \text{ in } Q,\\\displaystyle
d_1\partial_n\varphi_1 + \lambda_1\varphi_1 = 0 \ \text{ on } \Sigma,\quad \varphi_1(T) = \frac{\partial l}{\partial y_1}(x,y_u(T))\ \text{ in } \Omega,\end{array}\right.\label{E3.5}\\
&\left\{\begin{array}{l}\displaystyle -\frac{\partial\varphi_2}{\partial t} - d_2\Delta\varphi_2 - \dive{(V_2\varphi_2)}\\
\phantom{-\frac{\partial\varphi_1}{\partial t}}\displaystyle + (b_2 -2\gamma y_{u,1}y_{u,2})\varphi_2 + (a_2 + 2\gamma y_{u,1}y_{u,2})\varphi_1 = \frac{\partial L}{\partial y_2}(x,t,y_u)\ \text{ in } Q,\\\displaystyle
d_2\partial_n\varphi_2 + \lambda_2\varphi_2 = 0 \ \text{ on } \Sigma,\quad \varphi_2(T) = \frac{\partial l}{\partial y_2}(x,y_u(T))\ \text{ in } \Omega.\end{array}\right.\label{E3.6}
\end{align}
\label{T3.1}
\end{theorem}

In the theorem above, we have used the notation
\begin{align*}
&\frac{\partial^2L}{\partial y^2}(x,t,y_u)z^2_{u,v}\\
&= \frac{\partial^2L}{\partial y_1^2}(x,t,y_u)z^2_{u,v,1} + 2\frac{\partial^2L}{\partial y_1\partial y_2}(x,t,y_u)z_{u,v,1}z_{u,v,2} + \frac{\partial^2L}{\partial y_2^2}(x,t,y_u)z^2_{u,v,2}.
\end{align*}
The same notation is used for the expression $\frac{\partial^2l}{\partial y^2}(x,y_u(T))z_{u,v}^2$.

The proof of existence and uniqueness of a solution in $W(0,T)$ to the system \eqref{E3.5}--\eqref{E3.6} is the same as that for the system \eqref{E2.6}--\eqref{E2.7}. Indeed, one system is the adjoint of the other. The regularity in $L^\infty(Q)^2$ follows from \cite[Chapter III-\S 7]{Lad-Sol-Ura68} and the assumption \eqref{E3.1}. The formulas \eqref{E3.3} and \eqref{E3.4} follow from \eqref{E2.6}--\eqref{E2.7} and \eqref{E2.8}--\eqref{E2.9}, the chain rule, \eqref{E3.5}--\eqref{E3.6}, and integration by parts. Observe that, because of the terminal term in the functional $J$, it is not possible to deduce $C(\bar Q)$-regularity for $\varphi_u$.

\begin{theorem}
The problem \Pb has at least one solution.
\label{T3.2}
\end{theorem}

\begin{proof}
It was assumed that $\Aad \neq \emptyset$. Moreover, $\Aad$ is bounded in $L^\infty(\Sigma_c)$. From Theorem \ref{T2.1} we also deduce that $\Aad$ is closed in $L^\infty(\Sigma_c)$. Then, there exists a minimizing sequence $\{u_k\}_{k = 1}^\infty$ such that $u_k \stackrel{*}{\rightharpoonup} \bar u$ in $L^\infty(\Sigma_c)$ with $\bar u \in \Aad$. Since $y_{u_k} \to y_{\bar u}$ in $C(\bar Q)^2$, we deduce the existence of $M > 0$ such that $\|y_{u_k}\|_{C(\bar Q)^2} \le M$ for every $k \ge 1$. From Assumption \ref{A3.1} and the mean value theorem we find that
\[
|L(x,t,y_{u_k}(x,t))| \le |L(x,t,0)| + \psi_M(x,t)M\text{ and } |l(x,y_u(x,T))| \le |l(x,0)| + C_MM.
\]
Then, we deduce with Lebesgue's dominated convergence theorem and the lower semicontinuity of the mapping $u \to \|u\|_{L^2(\Sigma_c)^2}$ that $J(\bar u) \le \liminf_{k \to \infty}J(u_k)$. Consequently, $\bar u$ is a solution of \Pb.
\end{proof}

In the sequel, we say that $\bar u$ is a local minimizer or local solution of \Pb if there exists $\varepsilon > 0$ such that $J(\bar u) \le J(u)$ for every $u \in \Aad \cap B_\varepsilon(\bar u)$, where $B_\varepsilon(\bar u)$ denotes the closed ball of $L^2(\Sigma_c)$ centered at $\bar u$ and radius $\varepsilon$. If the inequality is strict for $u \neq \bar u$, we say that $\bar u$ is a strict local minimizer. The rest of this section is dedicated to the proof of the first-order optimality conditions satisfied by a local minimizer. For this purpose we use the following abstract result whose proof can be found in \cite[Theorem 5.2]{Casas1993}.

\begin{theorem}
Let $U$ and $Z$ be two Banach spaces, and $K \subset U$ and $C \subset Z$ two convex sets, $C$ having nonempty interior. Let $\bar u$ be a solution of the optimization problem
\[
(\mathcal{P})\quad \min_{u \in K,\ \mG(u) \in C} \mJ(u),
\]
where $\mJ:U \longrightarrow (-\infty,+\infty]$ and $\mG:U \longrightarrow Z$ are two G\^ateaux differentiable mappings at $\bar u$. If the following linearized Slater condition
\[
\exists u_0 \in K \text{ such that } \mG(\bar u) + D\mG(\bar u)(u_0 - \bar u) \in \text{int}(C)
\]
holds, then there exists $\bar\mu \in Z^*$ such that
\begin{align*}
&\langle\bar\mu,z - \mG(\bar u)\rangle \le 0\quad \forall z \in C,\\
&\langle \mJ'(\bar u) + [D\mG(\bar u)]^*\bar\mu,u - \bar u\rangle \ge 0\quad \forall u \in K.
\end{align*}
\label{T3.3}
\end{theorem}

Before applying this theorem, we introduce some notation. We denote by $\M(\bar Q)$ the Banach space of real and regular Borel measures on $\bar Q$, which is identified with the dual of $C(\bar Q)$. Moreover, we have
\[
\|\mu\|_{\M(\bar Q)} =\sup_{\|y\|_{C(\bar Q)} \le 1}\int_{\bar Q}y\dmu = |\mu|(\bar Q),
\]
where $|\mu|$ stands for the total variation measure associated with $\mu$. If we consider the Jordan decomposition of $\mu = \mu^+ - \mu^-$, then we know that $|\mu| = \mu^+ + \mu^-$. We refer the reader to \cite[Chapter 6]{Rudin1987} for these issues.

\begin{corollary}
Let $\bar u$ be a local solution of \Pb with associated state $\bar y$. Assume that the following Slater condition holds:
\begin{equation}
\exists u_0 \in \Uad : \ymin{i} < (\bar y_i+ z_{\bar u,u_0 - \bar u,i})(x,t) < \ymax{i},\, i = 1, 2, \ \forall (x,t) \in \bar Q,
\label{E3.7}
\end{equation}
where $z_{\bar u,u_0 - \bar u} = DG(\bar u)(u_0 - \bar u)$. Then, there exists $\bar\mu \in \M(\bar Q)^2$ such that
\begin{align}
&\int_{\bar Q}(y - \bar y)\dbmu = \sum_{i = 1}^2\int_{\bar Q}(y_i - \bar y_i)\dbmu_i \le 0\quad \forall y \in \K, \label{E3.8}\\
&J'(\bar u)(u - \bar u) + \int_{\bar Q}z_{\bar u,u - \bar u}\dbmu \ge 0\quad \forall u \in \Uad.\label{E3.9}
\end{align}
\label{C3.1}
\end{corollary}
The result follows directly from  Theorem \ref{T3.3} by taking $U = L^\infty(\Sigma_c)$, $Z = C(\bar Q)^2$, $K = \Uad$, $C = \K$, $\mG = G$, and $\mJ(u) = J(u) + I_{ B_{\varepsilon}(\bar u)}(u)$, and $I_{B_{\varepsilon}(\bar u)}$ is the indicator function of the closed ball $B_{\varepsilon}(\bar u)$ of $L^2(\Sigma_c)$: $I_{B_{\varepsilon}(\bar u)}(u) = 0$ if $u \in B_{\varepsilon}(\bar u)$ and $+\infty$ otherwise.
\begin{remark}
If we consider the Jordan decomposition of $\bar\mu = \bar\mu^+ - \bar\mu^-$, we infer from \eqref{E3.8} that
\begin{equation}
\left\{\begin{array}{l}\supp{\bar\mu_i^+} \subset \{(x,t) \in \bar Q : \bar y_i(x,t) = \ymax{i}\},\\\supp{\bar\mu_i^-} \subset \{(x,t) \in \bar Q : \bar y_i(x,t) = \ymin{i}\},\end{array}\right. i = 1, 2;
\label{E3.10}
\end{equation}
see, for instance, \cite[Proposition 2.5]{CMV2014}. Therefore, from Assumption \ref{A3.1} about $y_0$ we infer the existence of $t_0 \in (0,T]$ such that $\ymin{i} < \bar y_i(x,t) < \ymax{i}$ for every $(x,t) \in \bar\Omega \times [0,t_0)$, $i = 1, 2$. Consequently, we have that $\supp{\bar\mu_i} \subset \bar\Omega \times [t_0,T]$, $i = 1, 2$.
\label{R3.1}
\end{remark}

Associated with a measure $\mu \in \M(\bar Q)^2$ with support in $\bar\Omega \times (0,T]$, we introduce the following adjoint state equation
\begin{align}
&\left\{\begin{array}{l}\displaystyle -\frac{\partial\varphi_1}{\partial t} - d_1\Delta\varphi_1 - \dive{(V_1\varphi_1)}\\\phantom{-\frac{\partial\varphi_1}{\partial t}}\displaystyle + (\gamma y_{u,2}^2 + a_1)\varphi_1 + (b_1 - \gamma y_{u,2}^2)\varphi_2 = \mu_{1,Q}\ \text{ in } Q,\\
d_1\partial_n\varphi_1 + \lambda_1\varphi_1 = \mu_{1,\Sigma} \ \text{ on } \Sigma,\quad \varphi_1(T) = \mu_{1,\bar\Omega}\ \text{ in } \bar\Omega,\end{array}\right.\label{E3.11}\\
&\left\{\begin{array}{l}\displaystyle -\frac{\partial\varphi_2}{\partial t} - d_2\Delta\varphi_2 - \dive{(V_2\varphi_2)}\\
\phantom{-\frac{\partial\varphi_1}{\partial t}}\displaystyle + (b_2 -2\gamma y_{u,1}y_{u,2})\varphi_2 + (a_2 + 2\gamma y_{u,1}y_{u,2})\varphi_1 = \mu_{2,Q} \text{ in } Q,\\
d_2\partial_n\varphi_2 + \lambda_2\varphi_2 = \mu_{2,\Sigma} \ \text{ on } \Sigma,\quad \varphi_2(T) = \mu_{2,\bar\Omega}\ \text{ in } \bar\Omega.\end{array}\right.\label{E3.12}
\end{align}
Above, we have considered the decomposition $\mu_i = \mu_{i,Q} + \mu_{i,\Sigma} + \mu_{i,\bar\Omega}$, where $\mu_{i,Q}$, $\mu_{i,\Sigma}$, and $\mu_{i,\bar\Omega}$ denote the restrictions of $\mu_i$ to $Q$, $\Sigma$, and $\bar\Omega$, $i = 1, 2$. 

\begin{definition}
Given $\sigma \in \Big(1,\frac{n}{n - 1}\Big)$, we say that $\varphi_{\mu} \in L^1(0,T;W^{1,\sigma}(\Omega))^2$ is a solution of \eqref{E3.11}--\eqref{E3.12} if
\begin{align}
&\int_Q\varphi_{\mu,1}\Big[\frac{\partial z_1}{\partial t} - d_1\Delta z_1 + V_1\cdot\nabla z_1 + (a_1 + \gamma y_{u,2}^2)z_1\label{E3.13}\\
& + (a_2 + 2\gamma y_{u,1}y_{u,2})z_2\Big]\dx\dt + \int_\Sigma(d_1\partial_n z_1 + \lambda_1z_1)\varphi_{\mu,1}\dx\dt  = \int_{\bar Q}z_1\dbmu_1,\notag\\
&\int_Q\varphi_{\mu,2}\Big[\frac{\partial z_2}{\partial t} - d_2\Delta z_2 + V_2\cdot\nabla z_2 + (b_2 - 2\gamma y_{u,1}y_{u,2})z_2\label{E3.14}\\
&+ (b_1 - \gamma y_{u,2}^2)z_1\Big]\dx\dt + \int_\Sigma(d_2\partial_n z_2 + \lambda_2z_2)\varphi_{\mu,2}\dx\dt = \int_{\bar Q} z_2 \dbmu_2,\notag
\end{align}
for every $z \in \mZ$, where
\begin{align*}
\mZ = \{z \in [W(0,T) \cap C(\bar Q)]^2 : &\frac{\partial z_i}{\partial t} - d_i\Delta z_i + V_i\cdot\nabla z_i \in L^\infty(Q),\\
&\ \partial_nz_i \in L^\infty(\Sigma),  \text{ and } z_i(0) = 0,\ i= 1, 2\}.
\end{align*}
\label{D3.1}
\end{definition}

\begin{theorem}
The system \eqref{E3.11}-\eqref{E3.12} has a unique solution $\varphi_{\mu}$. Moreover, $\varphi_{\mu} \in L^{\sigma_1}(0,T;W^{1,\sigma_2}(\Omega))^2$ for every $\sigma_1, \sigma_2 \in [1,2)$ satisfying $\frac{2}{\sigma_1} + \frac{n}{\sigma_2} > n + 1$ and there exist constants $C_{\sigma_1,\sigma_2}$ independent of $\mu$ such that the following estimates hold
\begin{equation}
\|\varphi_{\mu}\|_{L^{\sigma_1}(0,T;W^{1,\sigma_2}(\Omega))} \le C_{\sigma_1,\sigma_2}\|\mu\|_{\M(\bar Q)}.
\label{E3.15}
\end{equation}
\label{T3.4}
\end{theorem}

The proof of this theorem is similar to that of Theorem 6.3 in \cite{Casas-1997} or Theorem 2.2 in \cite{CasasKunisch2016}. Now, the optimality conditions \eqref{E3.8}--\eqref{E3.9} for a local minimizer can be written as follows.

\begin{corollary}
Let $\bar u$ be a local minimizer of \Pb. Then, there exists $\bar y \in [W(0,T) \cap C(\bar Q)]^2$, $\bar\varphi \in L^{\sigma_1}(0,T;W^{1,\sigma_2}(\Omega))^2$ for every $\sigma_1, \sigma_2 \in [1,2)$ satisfying $\frac{2}{\sigma_1} + \frac{n}{\sigma_2} > n + 1$, and $\bar\mu \in \M(\bar Q)^2$ such that
\begin{align}
&\left\{
\begin{array}{l}
 \displaystyle \frac{\partial \bar y_1}{\partial t} - d_1\Delta\bar y_1 
 + V_{1}\cdot\nabla\bar y_1 + \gamma\bar y_1\bar y_2^2 + a_1\bar y_1 + a_2\bar y_2 = f_1\text{ in }Q, \\
  d_1\partial_n\bar y_1 + \lambda_1\bar y_1 = g_1\text{ on }\Sigma,\qquad \bar y_1(0) = y_{0,1}\text{ in }\Omega, 
\end{array}
\right.
\label{E3.16} \\
&\left\{
\begin{array}{l}
\displaystyle  \frac{\partial\bar y_2}{\partial t} - d_2\Delta\bar y_2 
+ V_{2}\cdot\nabla\bar y_2 - \gamma\bar y_1\bar y^2_2 + b_1\bar y_1 + b_2\bar y_2 = f_2\text{ in }Q, \\
  d_2\partial_n\bar y_2 +\lambda_2\bar y_2 = g_2 + \bar u\chi_{\Sigma_c}\text{ on }\Sigma,\qquad \bar y_2(0) = y_{0,2}\text{ in }\Omega, 
\end{array}
\right.\label{E3.17} \\
&\left\{\begin{array}{l}\displaystyle -\frac{\partial\bar\varphi_1}{\partial t} - d_1\Delta\bar\varphi_1 - \dive{(V_1\bar\varphi_1)}\\\phantom{-\frac{\partial\bar\varphi_1}{\partial t}}\displaystyle + (\gamma\bar y_2^2 + a_1)\bar\varphi_1 + (b_1 - \gamma\bar y_2^2)\bar\varphi_2 = \frac{\partial L}{\partial y_1}(x,t,\bar y) + \bar\mu_{1,Q}\ \text{ in } Q,\\\displaystyle
d_1\partial_n\bar\varphi_1 + \lambda_1\bar\varphi_1 = \bar\mu_{1,\Sigma} \ \text{ on } \Sigma,\quad \bar\varphi_1(T) = \frac{\partial l}{\partial y_1}(x,\bar y(T)) + \bar\mu_{1,\bar\Omega}\ \text{ in } \bar\Omega,\end{array}\right.\label{E3.18} \\
&\left\{\begin{array}{l}\displaystyle -\frac{\partial\bar\varphi_2}{\partial t} - d_2\Delta\bar\varphi_2 - \dive{(V_2\bar\varphi_2)}\\
\phantom{-\frac{\partial\varphi_1}{\partial t}}\displaystyle + (b_2 -2\gamma \bar y_1\bar y_2)\bar\varphi_2 + (a_2 + 2\gamma\bar y_1\bar y_2)\bar\varphi_1 = \frac{\partial L}{\partial y_2}(x,t,\bar y) + \bar\mu_{2,Q}\ \text{ in } Q,\\\displaystyle
d_2\partial_n\bar\varphi_2 + \lambda_2\bar\varphi_2 = \bar\mu_{2,\Sigma} \ \text{ on } \Sigma,\quad \bar\varphi_2(T) = \frac{\partial l}{\partial y_2}(x,\bar y(T)) + \bar\mu_{2,\bar\Omega}\ \text{ in } \bar\Omega,\end{array}\right.\label{E3.19}
\end{align}

\begin{align}
&\int_{\bar Q}(y - \bar y)\dbmu \le 0\quad \forall y \in \K, \label{E3.20}\\
&\int_{\Sigma_c}(\bar\varphi_2 + \nu\bar u)(u - \bar u)\dx\dt \ge 0\quad \forall u \in \Uad. \label{E3.21}
\end{align}
\label{C3.2}
\end{corollary}

We note that $\bar y = G(\bar u)$ is the state associated with $\bar u$. We also have that $\bar\varphi = \varphi_{\bar u} + \varphi_{\bar\mu}$ with $\varphi_{\bar u}$ solution of \eqref{E3.5}--\eqref{E3.6} and $\varphi_{\bar\mu}$ solution of \eqref{E3.11}--\eqref{E3.12} for $y_u = \bar y$ in both cases and $\mu = \bar\mu$. The inequality \eqref{E3.20} is the same as \eqref{E3.8}. Finally, \eqref{E3.21} follows from \eqref{E3.9}, \eqref{E3.3}, and \eqref{E3.14}.

\section{Second-order sufficient optimality conditions for a local minimizer}
\label{S4}
\setcounter{equation}{0}

In this section, $\bar u \in \Aad$ is a control satisfying the first-order optimality conditions \eqref{E3.16}--\eqref{E3.21} together with $\bar y$, $\bar\varphi$ and $\bar\mu$. The goal is to formulate a second-order condition for $\bar u$ to be a strict local minimizer of \Pb. For this purpose, we take $n + 1 < \sigma < \infty$ and define the Lagrangian function
\begin{align*}
&\mL:L^\sigma(\Sigma_c) \times \M(\bar Q)^2 \longrightarrow \mathbb{R}\\
&\mL(u,\mu) = J(u) +\int_{\bar Q}y_u\dmu.
\end{align*}
Note that $y_u \in C(\bar Q)^2$ if $u \in L^\sigma(\Sigma_c)$ for $\sigma > n + 1$. From Theorem \ref{T3.1} and \ref{T2.2} we infer that $\mL$ is of class $C^2$ and
\begin{align}
&\frac{\partial\mL}{\partial u}(u,\mu)v = \int_{\Sigma_c}(\varphi_2 + \nu u)v\dx\dt,\label{E4.1}\\
&\frac{\partial^2\mL}{\partial u^2}(u,\mu)v^2 = \nu \int_{\Sigma_c}v^2\dx\dt + \int_\Omega\frac{\partial^2l}{\partial y^2}(x,y_u(T))z_{u,v}^2(T)\dx\label{E4.2}\\
&+\int_Q\Big[\frac{\partial^2L}{\partial y^2}(x,t,y_u)z^2_{u,v} {+} 2\gamma(\varphi_2 - \varphi_1)(2y_{u,2}z_{u,v,1}z_{u,v,2} + y_{u,1}z_{u,v,2}^2)\Big]\dx\dt,\notag
\end{align}
where $\varphi = \varphi_u + \varphi_\mu$ with $\varphi_u$ and $\varphi_{\mu}$ solutions of \eqref{E3.5}--\eqref{E3.6} and \eqref{E3.11}--\eqref{E3.12}, respectively. Above we follow the notation
\begin{align*}
&\frac{\partial^2L}{\partial y^2}(x,t,y_u)z^2_{u,v}\\ &= (z_{u,v,1},z_{u,v,2})\left(\begin{array}{ll}\displaystyle \frac{\partial^2L}{\partial y_1^2}(x,t,y_u) & \displaystyle \frac{\partial^2L}{\partial y_1\partial y_2}(x,t,y_u)\vspace{2mm} \\\displaystyle \frac{\partial^2L}{\partial y_2\partial y_1}(x,t,y_u) & \displaystyle \frac{\partial^2L}{\partial y_2^2}(x,t,y_u)\end{array}\right)\left(\begin{array}{c}z_{u,v,1}\\z_{u,v,2}\end{array}\right).
\end{align*}
The same notation is used for $\frac{\partial^2l}{\partial y^2}(x,y_u(T))z^2_{u,v}(T)$. For every $\tau > 0$ we consider the cone

\begin{align*}
C^\tau_{\bar u,\bar \mu} = \Big \{&v\in L^\sigma(\Sigma_c):
\ \frac{\partial\mL}{\partial u}(\bar u,\bar\mu)v\leq \tau\|v\|_{L^\sigma(\Sigma_c)},\\
& v(x,t)\geq 0\text{ if }\bar u(x,t)= \umin,\ v(x,t)\leq 0\text{ if }\bar u(x,t) = \umax,\\
& z_{\bar u,v,i}(x,t) \le\tau\|v\|_{L^\sigma(\Sigma_c)}\text{ if }\bar y_i(x,t) = \ymax{i},\\ & z_{\bar u,v,i}(x,t) \ge -\tau\|v\|_{L^\sigma(\Sigma_c)}\text{ if }\bar y_i(x,t) = \ymin{i}, \, i = 1, 2, \text{ and }\\
& \int_{\bar Q} z_{\bar u,v} \dbmu \ge -\tau \|v\|_{L^\sigma(\Sigma_c)}\Big \}.
\end{align*}

This cone is an extension of the following one, introduced in \cite{CRT2008}:

\begin{align*}
C_{\bar u,\bar \mu} = \Big \{&v\in L^\sigma(\Sigma_c):
\ \frac{\partial\mL}{\partial u}(\bar u,\bar\mu)v=0,\\
& v(x,t)\geq 0\text{ if }\bar u(x,t)=\umin,\ v(x,t)\leq 0\text{ if }\bar u(x,t)=\umax,\\
& z_{\bar u,v,i}(x,t) \le 0\text{ if }\bar y_i(x,t) = \ymax{i},\\ &z_{\bar u,v,i}(x,t) \ge 0\text{ if }\bar y_i(x,t) = \ymin{i},\, i = 1, 2,\\ &\text{ and }  \int_{\bar Q} z_{\bar u,v}\dbmu = 0 \Big \}.
\end{align*}
Indeed, if  $(\bar u,\bar y,\bar\varphi,\bar\mu)$ satisfies the first-order optimality conditions, we have that $C_{\bar u,\bar \mu} = C^0_{\bar u,\bar \mu} \subset C^\tau_{\bar u,\bar \mu}\ \forall \tau >0$.
In  \cite{Dunn98a} it was proved that the second-order condition based on the critical cone of the type $C_{\bar u,\bar\mu}$ is not enough for a sufficient optimality condition in some cases. Hence an extension of the cone was suggested; see also \cite{Maurer-Zowe79}.

\begin{theorem}
Assume that $(\bar u,\bar y,\bar\varphi,\bar\mu)$ satisfies the first-order optimality conditions \eqref{E3.16}--\eqref{E3.21} and suppose that there exist $\tau> 0$ and $\delta >0$ such that
\begin{equation}
\frac{\partial^2\mL}{\partial u^2}(\bar u,\bar\mu)v^2\geq \delta \|v\|^2_{L^2(\Sigma_c)}\ \forall v\in C^\tau_{\bar u,\bar \mu}.
\label{E4.3}
\end{equation}
Then, there exist $\varepsilon >0$ and $\kappa >0$ such that
\begin{equation}
J(\bar u) +\frac{\kappa}{2}\|u-\bar u\|^2_{L^2(\Sigma_c)}\leq J(u)\ \forall u\in\Aad \text{ such that }\|u-\bar u\|_{L^2(\Sigma_c)} \leq \varepsilon.
\label{E4.4}
\end{equation}
\label{T4.1}
\end{theorem}
\begin{proof}
We argue by contradiction. If the theorem is false, for every integer $k \ge 1$ there exists $u_k \in \Aad$ such that
\begin{equation}
\|u_k - \bar u\|_{L^2(\Sigma_c)} < \frac{1}{k} \text{ and } J(u_k) < J(\bar u) + \frac{1}{2k}\|u_k - \bar u\|^2_{L^2(\Sigma_c)}.
\label{E4.5}
\end{equation}
Since $\{u_k\}_{k = 1}^\infty \subset \Uad$, the boundedness of $\{u_k\}_{k = 1}^\infty$ in $L^\infty(\Sigma_c)$ and the convergence $u_k \to \bar u$ in $L^2(\Sigma_c)$ imply that $\rho_k = \|u_k - \bar u\|_{L^\sigma(\Sigma_c)} \to 0$ as $k \to \infty$. We set $v_k = \frac{1}{\rho_k}(u_k - \bar u)$. Taking a subsequence, denoted in the same way, we get that
\begin{equation}
\|v_k\|_{L^\sigma(\Sigma_c)} = 1\ \forall k \ge 1\ \text{ and }\ v_k \rightharpoonup v\text{ in } L^\sigma(\Sigma_c).
\label{E4.6}
\end{equation}
Note that the continuous embedding $L^\sigma(\Sigma_c) \subset L^2(\Sigma_c)$ yields the convergence $v_k \rightharpoonup v$ in $L^2(\Sigma_c)$. The rest of the proof is divided into four steps.\vspace{2mm}

{\em Step I - $v \in C_{\bar u,\bar\mu}$.} First, we observe that \eqref{E3.20} and the fact that $y_{u_k} \in \K$ imply that $\int_{\bar Q}(y_{u_k} - \bar y)\dbmu \le 0$. Combining this with \eqref{E4.5} we infer that
\begin{equation}
\mL(u_k,\bar\mu) < \mL(\bar u,\bar\mu) + \frac{1}{2k}\|u_k - \bar u\|^2_{L^2(\Sigma_c)}.
\label{E4.7}
\end{equation}
Applying the mean value theorem, we infer from this inequality for some $\theta_k \in (0,1)$ and $u_{\theta_k} = \bar u + \theta_k(u_k - \bar u)$ that
\[
\frac{\partial\mL}{\partial u}(u_{\theta_k},\bar\mu)(u_k - \bar u) < \frac{1}{2k}\|u_k - \bar u\|^2_{L^2(\Sigma_c)}.
\]
Dividing this expression by $\rho_k$ we find
\[
\frac{\partial\mL}{\partial u}(u_{\theta_k},\bar\mu)v_k < \frac{1}{2k}\rho_k \to 0\ \ \text{as } k \to \infty.
\]
We denote by $\varphi_k$ the adjoint state associated to $(u_{\theta_k},\bar\mu)$: $\varphi_k = \varphi_{u_{\theta_k}} + \varphi_{\bar\mu}$. To pass to the limit above in the left hand side we have to study the convergence $\varphi_k \to  \bar\varphi$. Setting $\phi_k = \varphi_k -  \bar\varphi = \varphi_{u_{\theta_k}} - \varphi_{\bar u}$ and subtracting the equations satisfied by $\varphi_{u_{\theta_k}}$ and $\varphi_{\bar u}$ we find
\begin{align*}
&\left\{\begin{array}{l}\displaystyle -\frac{\partial\phi_{k,1}}{\partial t} - d_1\Delta\phi_{k,1} - \dive{(V_1\phi_{k,1})} + (\gamma \bar y_2^2 + a_1)\phi_{k,1} + (b_1 - \gamma\bar y_2^2)\phi_{k,2}\\ \displaystyle
= \gamma(\bar y_2^2 - y_{k,2}^2)\varphi_{k,1} + \gamma(y_{k,2}^2 - \bar y_2^2)\varphi_{k,2} + \frac{\partial L}{\partial y_1}(x,t,y_k) - \frac{\partial L}{\partial y_1}(x,t,\bar y)\ \text{ in } Q,\\\displaystyle
d_1\partial_n\phi_{k,1} + \lambda_1\phi_{k,1} = 0 \ \text{ on } \Sigma,\quad \phi_{k,1}(T) = \frac{\partial l}{\partial y_1}(x,y_k(T)) - \frac{\partial l}{\partial y_1}(x,\bar y(T))\ \text{ in } \Omega,\end{array}\right.\\
&\left\{\begin{array}{l}\displaystyle -\frac{\partial\phi_{k,2}}{\partial t} - d_2\Delta\phi_{k,2} - \dive{(V_2\phi_{k,2})} + (b_2 -2\gamma\bar y_1\bar y_2)\phi_{k,2} + (a_2 + 2\gamma \bar y_1\bar y_2)\phi_{k,1}\\\phantom{\displaystyle -\frac{\partial\phi_{k,2}}{\partial t} - d_2\Delta\phi_{k,2}} \displaystyle= 2\gamma(y_{k,1}y_{k,2} - \bar y_1\bar y_2)\varphi_{k,2} + 2\gamma(\bar y_1\bar y_2 - y_{k,1}y_{k,2})\varphi_{k,1}\\\phantom{\displaystyle -\frac{\partial\phi_{k,2}}{\partial t} - d_2\Delta\phi_{k,2}}\displaystyle + \frac{\partial L}{\partial y_2}(x,t,y_k) - \frac{\partial L}{\partial y_2}(x,t,\bar y) \text{ in } Q,\\\displaystyle
d_2\partial_n\phi_{k,2} + \lambda_2\phi_{k,2} =  0\ \text{ on } \Sigma,\quad \phi_{k,2}(T) = \frac{\partial l}{\partial y_2}(x,y_k(T)) - \frac{\partial l}{\partial y_2}(x,\bar y(T))\ \text{ in } \Omega.\end{array}\right.
\end{align*}
Since $u_{\theta_k} \to \bar u$ in $L^\sigma(\Sigma_c)$, we have that $y_k = G(u_{\theta_k}) \to G(\bar u)$ in $C(\bar Q)^2$. Using this property we infer with the aid of assumptions \eqref{E3.1} and \eqref{E3.2} that the right hand sides of the system above converge in $L^\infty(Q)$ or $L^\infty(\Omega)$ to 0. This implies that $\varphi_k - \bar\varphi \to 0$ in $[W(0,T) \cap L^\infty(Q)]^2$ as $k \to \infty$. Then, using \eqref{E4.1} and the above inequality we conclude that
\[
\frac{\partial\mL}{\partial u}(\bar u,\bar\mu)v \le \liminf_{k \to \infty}\frac{\partial\mL}{\partial u}(u_{\theta_k},\bar\mu)v_k \le 0.
\]
On the other hand, from \eqref{E3.21} and \eqref{E4.1} we infer
\[
\frac{\partial\mL}{\partial u}(\bar u,\bar\mu)v = \lim_{k \to \infty}\frac{1}{\rho_k}\frac{\partial\mL}{\partial u}(\bar u,\bar\mu)(u_k - \bar u) = \lim_{k \to \infty}\frac{1}{\rho_k}\int_{\Sigma_c}(\bar\varphi_2 + \nu\bar u)(u_k - \bar u)\dx\dt \ge 0.
\]
The last two inequalities yield $\frac{\partial\mL}{\partial u}(\bar u,\bar\mu)v = 0$. Since the set of elements $v \in L^\sigma(\Sigma_c)$ satisfying the second condition in the definition of $C_{\bar u,\bar\mu}$ is convex and closed, and every $v_k$ satisfies this second condition, we conclude that $v$ also does. Let us prove the third condition. Setting $\zeta_k = \frac{1}{\rho_k}(y_{u_k} - \bar y)$, subtracting the equations satisfied by $y_{u_k}$, and $\bar y$ and dividing by $\rho_k$ we get
\begin{align*}
&\left\{
\begin{array}{l}
 \displaystyle\frac{\partial \zeta_{k,1}}{\partial t} - d_1\Delta\zeta_{k,1} 
 + V_{1}\cdot\nabla\zeta_{k,1} + (a_1 + \gamma\bar y_2^2)\zeta_{k,1}\\ \phantom{\frac{\partial\zeta_{k,1}}{\partial t} - d_1\Delta\zeta_{k,1} } + [\gamma y_{u_k,1}(\bar y_2 + y_{u_k,2}) + a_2]\zeta_{k,2} = 0\text{ in }Q, \\
d_1\partial_n\zeta_{k,1} + \lambda_1\zeta_{k,1} = 0\text{ on }\Sigma,\qquad \zeta_{k,1}(0) = 0\text{ in }\Omega, 
\end{array}
\right.\\
&\left\{
\begin{array}{l}
\displaystyle  \frac{\partial\zeta_{k,2}}{\partial t} - d_2\Delta\zeta_{k,2} 
+ V_{2}\cdot\nabla\zeta_{k,2} + [b_2 - \gamma y_{u_k,1}(\bar y_2 + y_{u_k,2})]\zeta_{k,2}\\ \phantom{\frac{\partial\zeta_{k,1}}{\partial t} - d_1\Delta\zeta_{k,1} }  + (b_1 - \gamma \bar y^2_2)\zeta_{k,1} = 0\text{ in }Q, \\d_2\partial_n\zeta_{k,2} + \lambda_2\zeta_{k,2} = v_k\chi_{\Sigma_c}\text{ on }\Sigma,\qquad \zeta_{k,2}(0) = 0\text{ in }\Omega. 
\end{array}
\right.
\end{align*}
Since $\{v_k\}_{k = 1}^\infty$ is bounded in $L^\sigma(\Sigma_c)$, we infer from \cite[Chapter III]{Lad-Sol-Ura68} and \cite{DiBenedetto86} that $\{\zeta_k\}_{k = 1}^\infty$ is bounded in a H\"older space and in $W(0,T)^2$. Then, we can pass to the limit in the system above and deduce that $\|\zeta_k - z_{\bar u,v}\|_{C(\bar Q)} \to 0$ as $k \to \infty$. Therefore, due to $\zeta_{k,i}(x,t) \le 0\quad \text{if } \bar y_i(x,t) = \ymax{i}$, $i = 1, 2$, passing to the limit we infer that $z_{\bar u,v,i}(x,t) \le 0$ if $\bar y_i(x,t) = \ymax{i}$, $i = 1, 2$. Analogously, we prove that $z_{\bar u,v.i}(x,t) \ge 0$ if $\bar y_i(x,t) = \ymin{i}$, $i = 1, 2$.

Using the established properties of $z_{\bar u,v}$ and \eqref{E3.10} we obtain
\begin{align*}
0 = &\frac{\partial\mL}{\partial u}(\bar u,\bar\mu)v = J'(\bar u)v + \int_{\bar Q}z_{\bar u,v}\dbmu =  J'(\bar u)v + \int_{\bar Q}z_{\bar u,v}\dbmu^+ - \int_{\bar Q}z_{\bar u,v}\dbmu^-  \\
= & J'(\bar u)v - \int_{\bar Q}|z_{\bar u,v}|\dbmu^+ - \int_{\bar Q}|z_{\bar u,v}|\dbmu^- = J'(\bar u)v - \int_{\bar Q}|z_{\bar u,v}|\dabmu.
\end{align*}
Moreover, from \eqref{E4.5} we deduce with the mean value theorem that $J'(\bar u)v \le 0$. Then, the above equality is possible if and only if $J'(\bar u)v = \int_{\bar Q}z_{\bar u,v}\dbmu = 0$. This completes the proof of {\em Step I.}\vspace{2mm}

{\em Step II - There exists an integer $k_\tau \ge 1$ such that $\{v_k\}_{k = k_\tau}^\infty \subset C^\tau_{\bar u,\bar\mu}$.} Since $\displaystyle\lim_{k \to \infty}\mL(\bar u,\bar\mu)v_k = \mL(\bar u,\bar\mu)v = 0$ and recalling that $\|v_k\|_{L^\sigma(\Sigma_c)} = 1$, we infer the existence of $k_1$ such that $\mL(\bar u,\bar\mu)v_k \le \tau = \tau\|v_k\|_{L^\sigma(\Sigma_c)}$ for every $k \ge k_1$. Due to $\{u_k\}_{k = 1}^\infty \subset \Uad$, we have that $v_k(x,t) \ge 0$ if $\bar u(x,t) = \umin$ and $v_k(x,t) \le 0$ if $\bar u(x,t) = \umax$. The inequality $\zeta_{k,i}(x,t) \le 0$ if $\bar y_i(x,t) = \ymax{i}$, $i = 1, 2$, and the convergence $\zeta_k \to z_{\bar u,v}$ in $C(\bar Q)^2$ imply the existence of $k_2$ such that $z_{\bar u,v,i}(x,t) \le \tau = \tau\|v_k\|_{L^\sigma(\Sigma_c)}$ if $\bar y_i(x,t) = \ymax{i}$ for all $k \ge k_2$. In a similar way we get the existence of $k_3$ such that $z_{\bar u,v,i}(x,t) \ge -\tau = -\tau\|v_k\|_{L^\sigma(\Sigma_c)}$ if $\bar y_i(x,t) = \ymin{i}$ for all $k \ge k_3$. Finally, the convergence $\int_{\bar Q}z_{\bar u,v_k}\dbmu \to \int_{\bar Q}z_{\bar u,v}\dbmu = 0$ yields the existence of $k_4$ such that $\int_{\bar Q}z_{\bar u,v_k}\dbmu \ge -\tau = -\tau\|v_k\|_{L^\sigma(\Sigma_c)}$ for every $k \ge k_4$. Hence, it is enough to set $k_\tau = \max_{1 \le j\le 4}k_j$ to conclude that $\{v_k\}_{k = k_\tau}^\infty \subset C_{\bar u,\bar\mu}^\tau$.\vspace{2mm}

{\em Step III - Contradiction.} Performing a second-order Taylor expansion of $\mL(u_k,\bar\mu)$ around $(\bar u,\bar\mu)$, taking into account that $\frac{\partial\mathcal{L}}{\partial u}(\bar u, \bar\mu)(u_k - \bar u)\ge 0$, and using \eqref{E4.7} we obtain
\begin{align*}
\frac{1}{2}\frac{\partial^2\mathcal{L}}{\partial u^2}(u_{\vartheta_k}, \bar\mu)(u_k - \bar u)^2 &\le \frac{\partial\mathcal{L}}{\partial u}(\bar u, \bar\mu)(u_k - \bar u) + \frac{1}{2}\frac{\partial^2\mathcal{L}}{\partial u^2}(u_{\vartheta_k}, \bar\mu)(u_k - \bar u)^2\\
&= \mL(u_k,\bar\mu) - \mL(\bar u,\bar\mu) < \frac{1}{2k}\|u_k - \bar u\|^2_{L^2(\Sigma_c)},
\end{align*}
for some $\vartheta_k \in (0,1)$ and $u_{\vartheta_k} = \bar u + \vartheta_k(u_k - \bar u)$. Dividing the above inequality by $\rho_k^2$ we infer
\begin{equation}
\frac{\partial^2\mathcal{L}}{\partial u^2}(u_{\vartheta_k}, \bar\mu)v_k^2 < \frac{1}{k}\|v_k\|^2_{L^2(\Sigma_c)}\quad \forall k \ge k_\tau.
\label{E4.8}
\end{equation}
We next prove that there exists $k_\delta \ge 1$ such that
\begin{equation}
\Big|\frac{\partial^2\mathcal{L}}{\partial u^2}(u_{\vartheta_k}, \bar\mu)v_k^2-\frac{\partial^2\mathcal{L}}{\partial u^2}(\bar u, \bar\mu)v_k^2\Big|\leq \frac{\delta}{2}\|v_k\|_{L^2(Q)}^2\quad \forall k\geq k_\delta.
\label{E4.9}
\end{equation}
Let us take $k_0 = \max\{k_\tau,k_\delta\}$. Since $\{v_k\}_{k = k_0}^\infty \in C^\tau_{\bar u,\bar\mu}$, \eqref{E4.3} yields
\[
\delta\|v_k\|^2_{L^2(\Sigma_c)} \le \frac{\partial^2\mathcal{L}}{\partial u^2}(\bar u, \bar\mu)v_k^2\quad \forall k \ge k_0.
\]
Combining this with \eqref{E4.9} and using \eqref{E4.8} we get
\[
\frac{\delta}{2}\|v_k\|^2_{L^2(\Sigma_c)} \le \frac{\partial^2\mathcal{L}}{\partial u^2}(u_{\vartheta_k}, \bar\mu)v_k^2 < \frac{1}{k}\|v_k\|^2_{L^2(\Sigma_c)}\qquad \forall k \ge k_0.
\]
Since $\|v_k\|_{L^\sigma(\Sigma_c)} = 1$, we have that $\|v_k\|^2_{L^2(\Sigma_c)} \neq 0$. Hence, the above inequality implies that $\frac{\delta}{2} < \frac{1}{k}$ for all $k \ge k_0$ which contradicts the fact that $\delta > 0$.\vspace{2mm}

{\em Step IV - Proof of \eqref{E4.9}.} Setting $\varphi_k = \varphi_{u_{\vartheta_k}} + \varphi_{\bar\mu}$, $y_k = y_{u_{\vartheta_k}}$, $z_k = z_{u_{\vartheta_k},v_k} = DG(u_{\vartheta_k})v_k$, and $z_{v_k} = z_{\bar u,v_k} = DG(\bar u)v_k$, we get with \eqref{E4.2}
\begin{align*}
&\Big|\frac{\partial^2\mathcal{L}}{\partial u^2}(u_{\vartheta_k}, \bar\mu)v_k^2-\frac{\partial^2\mathcal{L}}{\partial u^2}(\bar u, \bar\mu)v_k^2\Big| \le \int_\Omega\Big|\frac{\partial^2l}{\partial y^2}(x,y_k)z_k^2 - \frac{\partial^2l}{\partial y^2}(x,\bar y)z_{v_k}^2\Big|\dx\\
&+\int_Q\Big|\frac{\partial^2L}{\partial y^2}(x,t,y_k)z_k^2 - \frac{\partial^2L}{\partial y^2}(x,t,\bar y)z_{v_k}^2\Big|\dx\dt\\
&+ 2\gamma\int_Q|(\varphi_{u_{\vartheta_k},2} - \bar\varphi_2) - (\varphi_{u_{\vartheta_k},1} - \bar\varphi_1)||2y_{k,2}z_{k,1}z_{k,2} + y_{k,1}z_{k,2}^2|\dx\dt\\
&+2\gamma\int_Q|\bar\varphi_2 - \bar\varphi_1|[2|y_{k,2}z_{k,1}z_{k,2} -\bar y_2z_{v_k,1}z_{v_k,2}| + |y_{k,1}z_{k,2}^2 - \bar y_1z_{v_k,2}^2|]\dx\dt\\
&= I_1 + I_2 + I_3 + I_4.
\end{align*}
To estimate $I_1$ and $I_2$ we need an estimate for $e_k = z_k - z_{v_k}$. Subtracting the equation for $z_k$ and $z_{v_k}$ we find
\begin{align*}
&\left\{\begin{array}{l}
 \displaystyle \frac{\partial e_{k,1}}{\partial t} - d_1\Delta e_{k,1} + V_{1}\cdot\nabla e_{k,1} + (\gamma \bar y_2^2 + a_1)e_{k,1} + (a_2 + 2\gamma\bar y_1\bar y_2)e_{k,2}\\
\qquad = \gamma(\bar y^2_2 - y_{k,2}^2)z_{k,1} + 2\gamma(\bar y_1\bar y_2 - y_{k,1}y_{k,2})z_{k,2}\text{ in }Q \\
d_1\partial_ne_{k,1} + \lambda_1e_{k,1} = 0\text{ on }\Sigma,\qquad e_{k,1}(0) = 0\text{ in }\Omega, 
\end{array}\right. \\
&\left\{\begin{array}{l}
\displaystyle  \frac{\partial e_{k,2}}{\partial t} - d_2\Delta e_{k,2} 
+ V_{2}\cdot\nabla e_{k,2} + (b_2 - 2\gamma\bar y_1\bar y_2)e_{k,2} + (b_1 - \gamma\bar y_2^2)e_{k,1}\\
\qquad = 2\gamma(\bar y_1\bar y_2 - y_{k,1}y_{k,2})z_{k,2} + \gamma(\bar y_2^2 - y_{k,2}^2)z_{k,1} \text{ in }Q \\ d_2\partial_ne_{k,2} +\lambda_2e_{k,2} = 0\text{ on }\Sigma,\qquad e_{k,2}(0) = 0\text{ in }\Omega, 
\end{array} \right.
\end{align*}
From this system we deduce the existence of a constant $C_1$ such that
\begin{equation}
\|z_k - z_{v_k}\|_{C(\bar Q)} \le C_1\rho_k\|v_k\|_{L^2(\Sigma_c)}\quad \forall k \ge 1.
\label{E4.10}
\end{equation}
Indeed, to prove this it is enough to get an estimate of the right hand side of the system above in $L^\infty(0,T;L^2(\Omega))$ in terms of $\rho_k\|v_k\|_{L^2(\Sigma_c)}$. We now prove this estimate. First, using Theorem \ref{T2.1}, and estimates \eqref{E2.3} and \eqref{E2.11} we obtain
\begin{align*}
&\|(\bar y^2_2 - y_{k,2}^2)z_{k,1}\|_{L^\infty(0,T;L^2(\Omega))} \le \|\bar y_2 - y_{k,2}\|_{L^\infty(0,T;L^2(\Omega))}\|\bar y_2 + y_{k,2}\|_{C(\bar Q)}\|z_{k,1}\|_{C(\bar Q)}\\
&\le C_2\|u_{\vartheta_k} - \bar u\|_{L^2(\Sigma_c)}\|v_k\|_{L^2(\Sigma_c)} = C_2\rho_k\|v_k\|^2_{L^2(\Sigma_c)}\\
&\le C_2|Q|^{\frac{\sigma - 2}{2\sigma}}\rho_k\|v_k\|_{L^\sigma(\Sigma_c)}\|v_k\|_{L^2(\Sigma_c)} = C_2|Q|^{\frac{\sigma - 2}{2\sigma}}\rho_k\|v_k\|_{L^2(\Sigma_c)}.
\end{align*}
Using again Theorem \ref{T2.1}, \eqref{E2.3}, and \eqref{E2.10} we infer
\begin{align*}
&\|(\bar y_1\bar y_2 - y_{k,1}y_{k,2})z_{k,2}\|_{L^\infty(0,T;L^2(\Omega))} \le \|\bar y_1 - y_{k,1}\|_{C(\bar Q)}\|\bar y_2\|_{C(\bar Q)}\|z_{k,2}\|_{L^\infty(0,T;L^2(\Omega))}\\
&\qquad\qquad\qquad+ \|\bar y_2 - y_{k,2}\|_{C(\bar Q)}\|y_{k,1}\|_{C(\bar Q)}\|z_{k,2}\|_{L^\infty(0,T;L^2(\Omega))}\\
&\le C_3\|u_{\vartheta_k} - \bar u\|_{L^2(\Sigma_c)}\|z_{k,2}\|_{L^\infty(0,T;L^2(\Omega))} + C_4\|u_{\vartheta_k} -  \bar u\|_{L^\sigma(\Sigma_c)}\|z_{k,2}\|_{L^\infty(0,T;L^2(\Omega))}\\
&\le C_5\rho_k\Big(\|v_k\|_{L^2(\Sigma_c)} + \|v_k\|_{L^\sigma(\Sigma_c)}\Big)\|v_k\|_{L^2(\Sigma_c)} \le C_6\rho_k\|v_k\|_{L^2(\Sigma_c)}.
\end{align*}
Thus, the last two estimates imply \eqref{E4.10}. Now, we estimate $I_1$. From the estimate \eqref{E4.10} and \eqref{E2.10} we get
\begin{align*}
I_1 &\le \int_\Omega\Big|\frac{\partial^2l}{\partial y^2}(x,y_k(x,T)) - \frac{\partial^2l}{\partial y^2}(x,\bar y(x,T))\Big|z^2_{v_k}(T)\dx\\
&+ \int_\Omega\Big|\frac{\partial^2l}{\partial y^2}(x,y_k(x,T))\Big||z_k(T) - z_{v_k}(T)||z_k(T) + z_{v_k}(T)|\dx\\
&\le \Big\|\frac{\partial^2l}{\partial y^2}(x,y_k(x,T)) - \frac{\partial^2l}{\partial y^2}(x,\bar y(x,T))\Big\|_{L^\infty(\Omega)}\|z_{v_k}(T)\|^2_{L^2(\Omega)}\\
&+ \Big\|\frac{\partial^2l}{\partial y^2}(x,y_k(x,T))\Big\|_{L^\infty(\Omega)}\|z_k(T) - z_{v_k}(T)\|_{L^2(\Omega)}\|z_k(T) + z_{v_k}(T)\|_{L^2(\Omega)}\\
&\le C_7\Big(\Big\|\frac{\partial^2l}{\partial y^2}(x,y_k(x,T)) - \frac{\partial^2l}{\partial y^2}(x,\bar y(x,T))\Big\|_{L^\infty(\Omega)} + \rho_k\Big)\|v_k\|^2_{L^2(\Sigma_c)}.
\end{align*}
Arguing similarly we prove that
\[
I_2 \le C_8\Big(\Big\|\frac{\partial^2L}{\partial y^2}(x,t,y_k) - \frac{\partial^2L}{\partial y^2}(x,t,\bar y)\Big\|_{L^\infty(Q)} + \rho_k\Big)\|v_k\|^2_{L^2(\Sigma_c)}.
\]
In order to estimate $I_3$ we use that $\varphi_k \to \bar\varphi$ in $[W(0,T) \cap L^\infty(Q)]^2$; see the proof of this in Step I. With this convergence property, the boundedness of $\{y_k\}_{k = 1}^\infty$ in $C(\bar Q)^2$, and \eqref{E2.10} we get
\[
I_3 \le C_9\|\varphi_k - \bar\varphi\|_{C(\bar Q)}\|z_k\|^2_{L^2(Q)} \le C_{10}\|\varphi_k - \bar\varphi\|_{C(\bar Q)}\|v_k\|^2_{L^2(\Sigma_c)}.
\]
To estimate $I_4$ we consider the following inequalities
\begin{align*}
&|y_{k,2}z_{k,1}z_{k,2} - \bar y_2z_{v_k,1}z_{v_k,2}|\\
&\qquad\le |y_{k,2} - \bar y_2||z_{k,1}||z_{k,2}| + |\bar y_2|(|z_{k,1} - z_{v_k,1}| |z_{k,2}| + |z_{k,2} - z_{v_k,2}|  |z_{v_k,1}|),\\
&|y_{k,1}z_{k,2}^2 - \bar y_1z_{v_k,2}^2| \le |y_{k,1} - \bar y_1|z_{k,2}^2 + |\bar y_1||z_{k,2} - z_{v_k,2}||z_{k,2} + z_{v_k,2}|.
\end{align*}
By Theorem \ref{T3.4} we know that $\bar\varphi \in L^{\sigma_1}(0,T;W^{1,\sigma_2}(\Omega))$ for every $\sigma_1, \sigma_2 \in [1,2)$ satisfying that $\frac{2}{\sigma_1} + \frac{n}{\sigma_2} > n + 1$. For every $\tilde\sigma \in [1,\frac{n + 2}{n})$ there exist $\sigma_1$ and $\sigma_2$ satisfying the previous conditions such that the continuous embedding $L^{\sigma_1}(0,T;W^{1,\sigma_2}(\Omega)) \subset L^{\tilde\sigma}(Q)$ holds. We select $\tilde\sigma = \frac{4}{3}$ and $\sigma_0 = 4$ in \eqref{E2.12}. Then we get with \eqref{E2.3}, \eqref{E2.11}, and \eqref{E2.12}
\begin{align*}
&\int_Q|\bar\varphi_2 - \bar\varphi_1|||y_{k,2} - \bar y_2||z_{k,1}||z_{k,2}|\dx\dt\\
&\le \|y_{k,2} - \bar y_2\|_{C(\bar Q)}\|z_{k,1}\|_{C(\bar Q)}\|\bar\varphi\|_{L^{\frac{4}{3}}(Q)}\|z_{k,2}\|_{L^4(Q)}\\
&\le C_{11}\|u_{\vartheta_k} - \bar u\|_{L^\sigma(\Sigma_c)}\|v_k\|^2_{L^2(\Sigma_c)} = C_{11}\rho_k\|v_k\|^2_{L^2(\Sigma_c)}.
\end{align*}
From \eqref{E2.12} and \eqref{E4.10} we deduce
\begin{align*}
&\int_Q|\bar\varphi_2 - \bar\varphi_1||\bar y_2|(|z_{k,1} - z_{v_k,1}| |z_{k,2}| + |z_{k,2} - z_{v_k,2}| |z_{v_k,1}|)\dx\dt \\
&\le \|\bar y_2\|_{C(\bar Q)}\|\bar\varphi\|_{L^{\frac{4}{3}}(Q)}\|z_k - z_{v_k}\|_{C(\bar Q)}\|z_{v_k}\|_{L^4(Q)} \le C_{12}\rho_k\|v_k\|^2_{L^2(\Sigma_c)}.
\end{align*}
With \eqref{E2.3} and \eqref{E2.12} we infer
\begin{align*}
&\int_Q|\bar\varphi_2 - \bar\varphi_1||y_{k,1} - \bar y_1|z_{k,2}^2\dx\dt \le \|\bar\varphi\|_{L^{\frac{4}{3}}(Q)}\|y_{k,1} - \bar y_1\|_{C(\bar Q)}\|z_{k,2}\|_{L^\infty(Q)}\|z_{k,2}\|_{L^4(Q)}\\
&\le C_{13}\|u_{\vartheta_k} - \bar u\|_{L^2(\Sigma_c)}\|v_k\|_{L^\sigma(\Sigma_c)}\|v_k\|_{L^2(\Sigma_c)} = C_{13}\rho_k\|v_k\|^2_{L^2(\Sigma_c)}.
\end{align*}
Finally, using again \eqref{E2.12} and \eqref{E4.10} it follows
\begin{align*}
&\int_Q|\bar\varphi_2 - \bar\varphi_1||\bar y_1||z_{k,2} - z_{v_k,2}||z_{k,2} + z_{v_k,2}|\dx\dt\\
&\le \|\bar\varphi\|_{L^{\frac{4}{3}}(Q)}\|\bar y_1\|_{C(\bar Q)}\|z_{k,2} - z_{v_k,2}\|_{C(\bar Q)}\|z_{k,2} + z_{v_k,2}\|_{L^4(Q)} \le C_{14}\rho_k\|v_k\|^2_{L^2(\Sigma_c)}.
\end{align*}

Adding the last four estimates we get that $I_4 \le \Big(\sum_{i = 11}^{14}C_i\Big)\rho_k\|v_k\|^2_{L^2(\Sigma_c)}$. Since $\lim_{k \to \infty}\|y_k - \bar y\|_{C(\bar Q)} = 0$ we infer with \eqref{E3.2} that
\begin{align*}
&\lim_{k \to \infty}\Big\|\frac{\partial^2l}{\partial y^2}(x,y_k(x,T)) - \frac{\partial^2l}{\partial y^2}(x,\bar y(x,T))\Big\|_{L^\infty(\Omega)} = 0,\\
&\lim_{k \to \infty}\Big\|\frac{\partial^2L}{\partial y^2}(x,t,y_k) - \frac{\partial^2L}{\partial y^2}(x,t,\bar y)\Big\|_{L^\infty(Q)} = 0.
\end{align*}
Recalling that $\lim_{k \to \infty}\|\varphi_k - \bar\varphi\|_{L^\infty(Q)} = 0$ and $\lim_{k \to \infty}\rho_k = 0$, from the estimates obtained for $\{I_i\}_{i = 1}^4$ we deduce the existence of $k_\delta \ge 1$ such that \eqref{E4.9} holds.
\end{proof}
\section{Examples}
\label{S5}
To illustrate the obtained theoretical results, we present two numerical experiments. In each of them, we solve a finite element discretization of the uncontrolled system $(u=0)$ starting from a small perturbation of a homogeneous steady state. We run the simulation until a non-homogeneous steady state is reached, up to machine precision, at time $T_\infty$. We denote this state by $y_d$. We then choose $T<T_\infty$. 
Our objective is to find a control $\bar u$ that drives the controlled state $\bar y$ close to $y_d(T_\infty)$ at time $T$.

We consider the Schnakenberg model. With our notation, the activator is $y_2$ and the inhibitor is $y_1$. In both examples, $\Omega = (0,1)^2$, $\lambda_1=\lambda_2 = 0$, $g_1(x,t)=g_2(x,t)=0$, $a_1=0$, $a_2=0$, $b_1=0$ and, for some positive parameters $a$, $b$, and $\gamma$, $b_2 = \gamma$, $f_1(x,t) = \gamma b$, $f_2(x,t) = \gamma a$. The homogeneous steady state for $u=0$ is given by $(b/(a+b)^2,a+b)$.

\paragraph{Example 1} We use the parameters from the example presented in \cite[Example 5]{Ruuth1995} and also considered in \cite{Fernandes-Fairweather2012,Garvie-Maini-Trenchea2010,Madzvamuse2006} and \cite[Example 9.2]{BenkovaPearsonPtashnyk2026}, which are $d_1 = 10$, $d_2=1$, $V_1 = V_2 = (0,0)$, $a=0.126779$, $b=0.792366$, and $\gamma=10^3$. The initial state is given by $y_{0,1}(x) = b/(a+b)^2 + p(x)$, $y_{0,2}(x) = a+b + p(x)$, where
\begin{align*}
    p(x) & =  0.0016\cos(2\pi(x_1+x_2)) + 0.01\sum_{j=1}^8\cos(2\pi j x_1).
\end{align*}
For the finite element approximation, we use a mesh of size $h=1/160=6.25\times 10^{-3}$ and a time step $\tau = 1/256\approx 3.9\times 10^{-3}$; computation details are given below.
In this case, we can stop at $T_\infty=0.5$. In Figure \ref{F3} (dashed red line), we show the $L^2(\Omega)$-distance between the final target state, shown in Figure \ref{F4}(b), and the intermediate uncontrolled states. 

For $T=0.125$, we solve the optimal control problem 
\[\min_{u\in \Aad} J(u) = \frac{1}{2} \|y_u(T) - y_d(T_\infty)\|^2_{L^2(\Omega)} + \frac{1}{2} \|u\|^2_{L^2(\Sigma)},\]
where the admissible set $\Aad$ is defined by $\umin = -1$, $\umax = 1$, $\ymin{1}=\ymin{2}=0.1$ and $\ymax{1}=\ymax{2}=+\infty$.
The controlled state at \(T=0.125\) is already very close to the target steady state; see Figure \ref{F3} blue line and Figure \ref{F4}(c). In this experiment, the state constraints are not active.

\begin{figure}
  \centering
  \includegraphics[width=.4\textwidth]{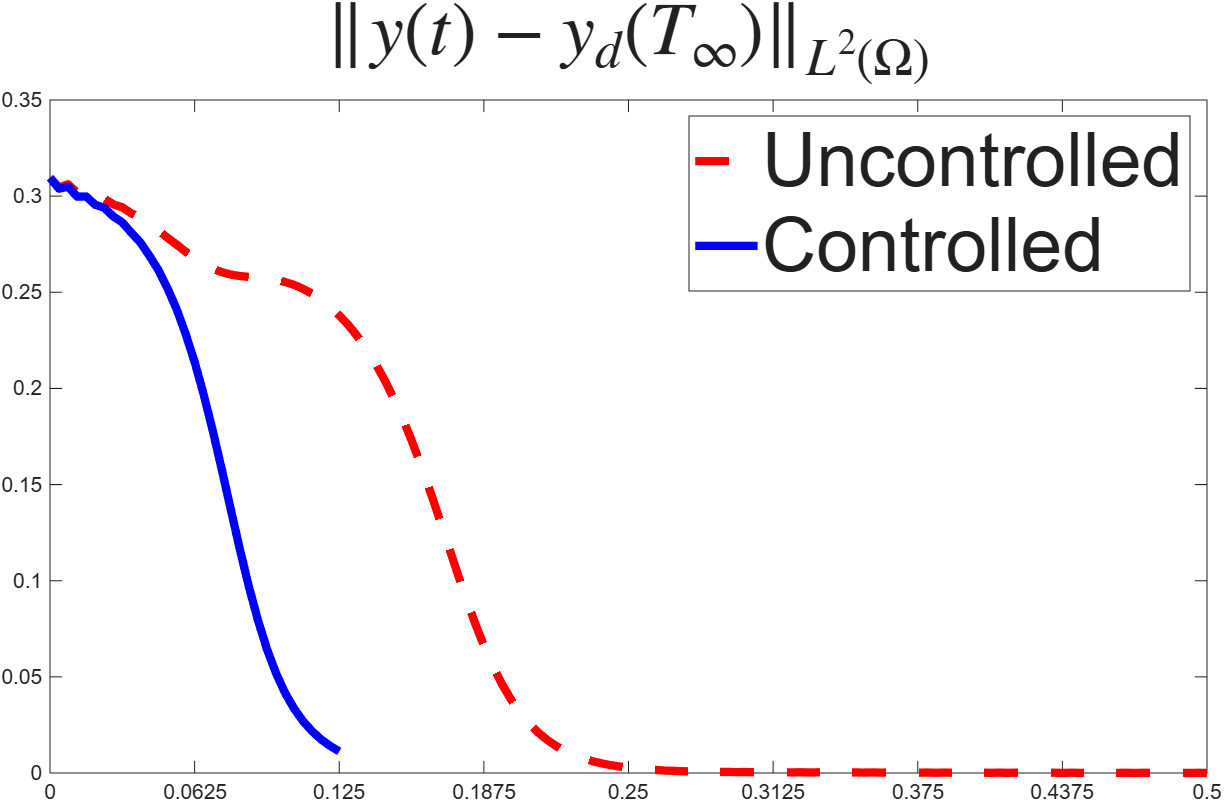}
  \caption{Example 1. }\label{F3}
\end{figure}

\begin{figure}
  \centering
  \begin{subfigure}[t]{0.2\textwidth}
  \includegraphics[width=\textwidth]{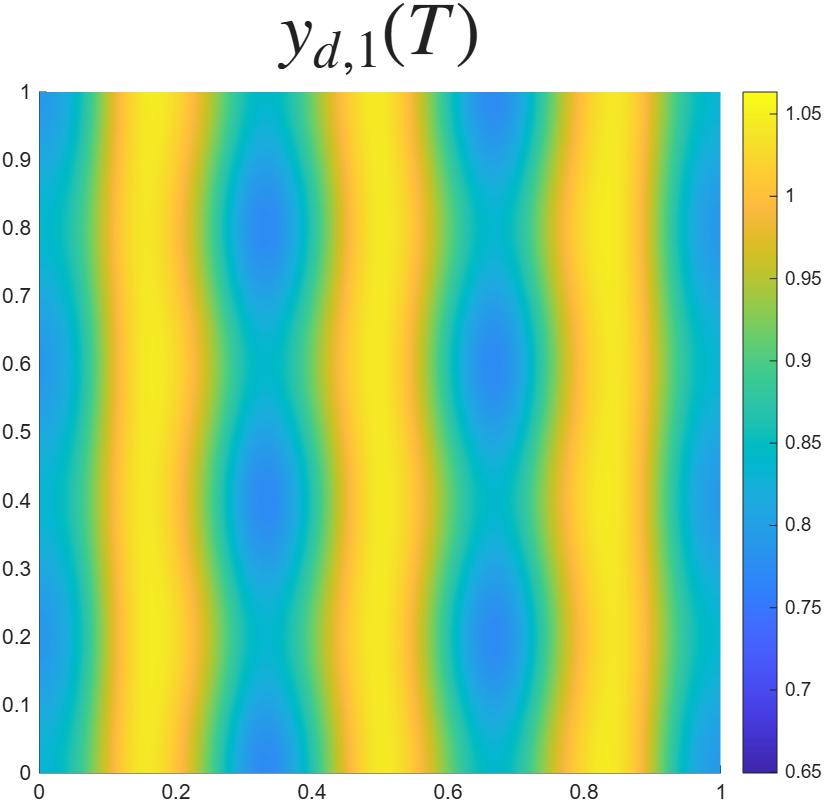}
  \includegraphics[width=\textwidth]{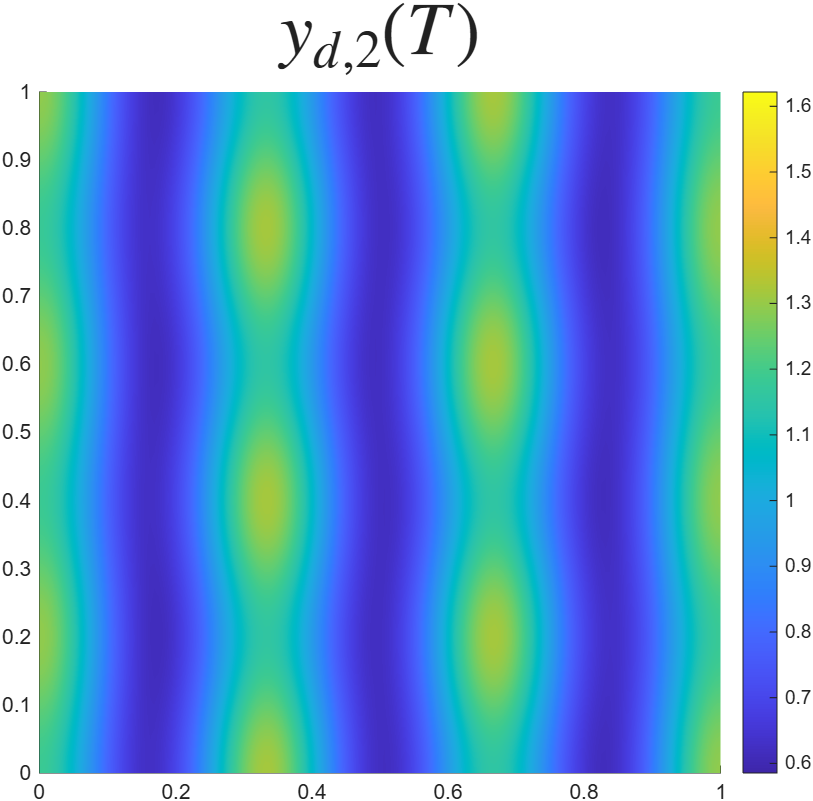}
  \caption{}
  \label{F4a}
  \end{subfigure}
  \begin{subfigure}[t]{0.2\textwidth}
  \includegraphics[width=\textwidth]{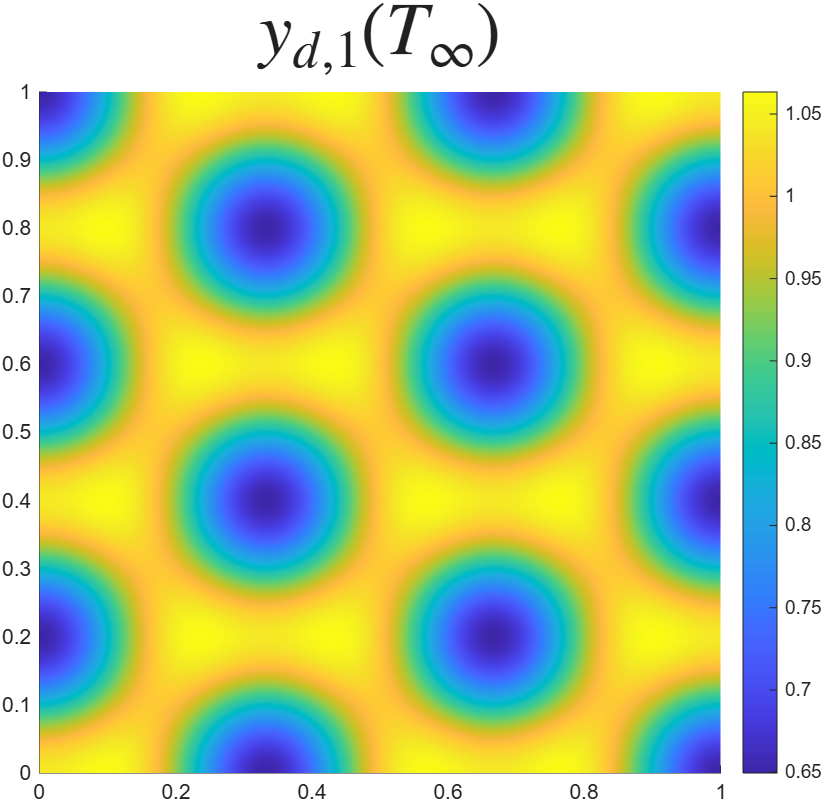}
   \includegraphics[width=\textwidth]{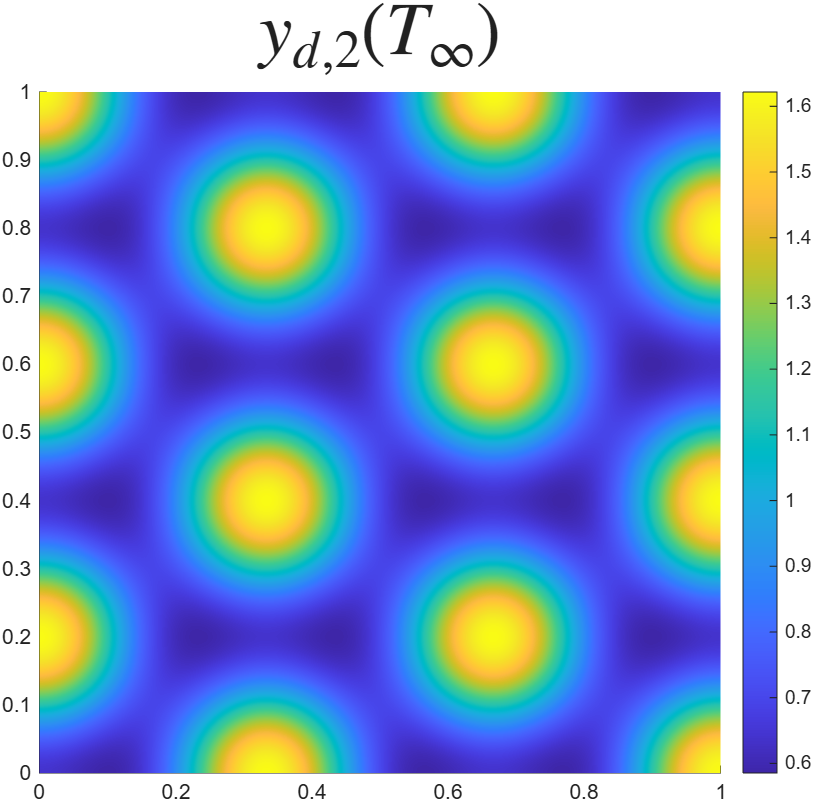}
  \caption{}
  \label{F4b}
  \end{subfigure}
  \begin{subfigure}[t]{0.2\textwidth}
  \includegraphics[width=\textwidth]{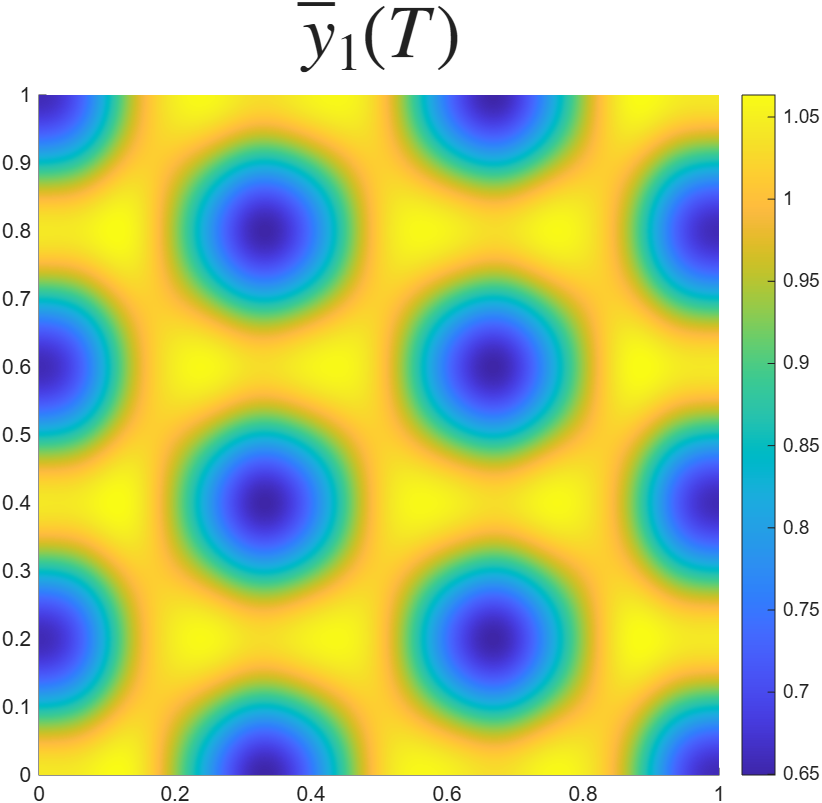}
   \includegraphics[width=\textwidth]{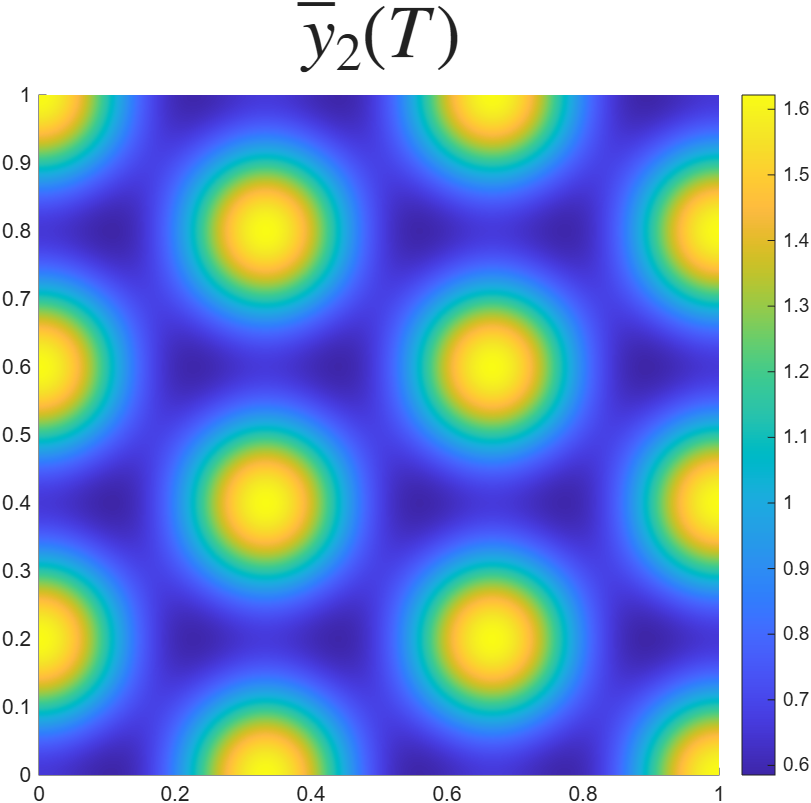}
  \caption{}
  \label{F4c}
  \end{subfigure}
    \caption{Example 1\\ {\small (\textsc{a}) uncontrolled state at $T=0.125$ (\textsc{b}) target state, (\textsc{c}) controlled state at $T=0.125$ }}
    \label{F4}
\end{figure}

\paragraph{Example 2} We consider a similar example illustrating the role of the state constraints. In this case,  $d_1 = 10^{-2}$, $d_2 = 10^{-4}$, $V_1 = V_2 = 0.01(\sin(\pi x_1),-\sin(\pi x_2))$, $a=0.01$, $b=2$, and $\gamma = 1$. The initial state is given by  $y_{0,1}(x) = b/(a+b)^2$, $y_{0,2}(x) = a+b + 5\times 10^{-3}\chi_\omega$, where $\omega=[3/8,5/8]^2$. In this example $h=0.125$, $\tau=0.125$, and $T_\infty=160$. The dashed red line in Figure \ref{F5} shows the $L^2(\Omega)$-distance between the final target state 
and the intermediate uncontrolled states. Snapshots of the activator at regular time intervals can be found in Figure \ref{F7}.

\begin{figure}
  \centering
  \includegraphics[width=.4\textwidth]{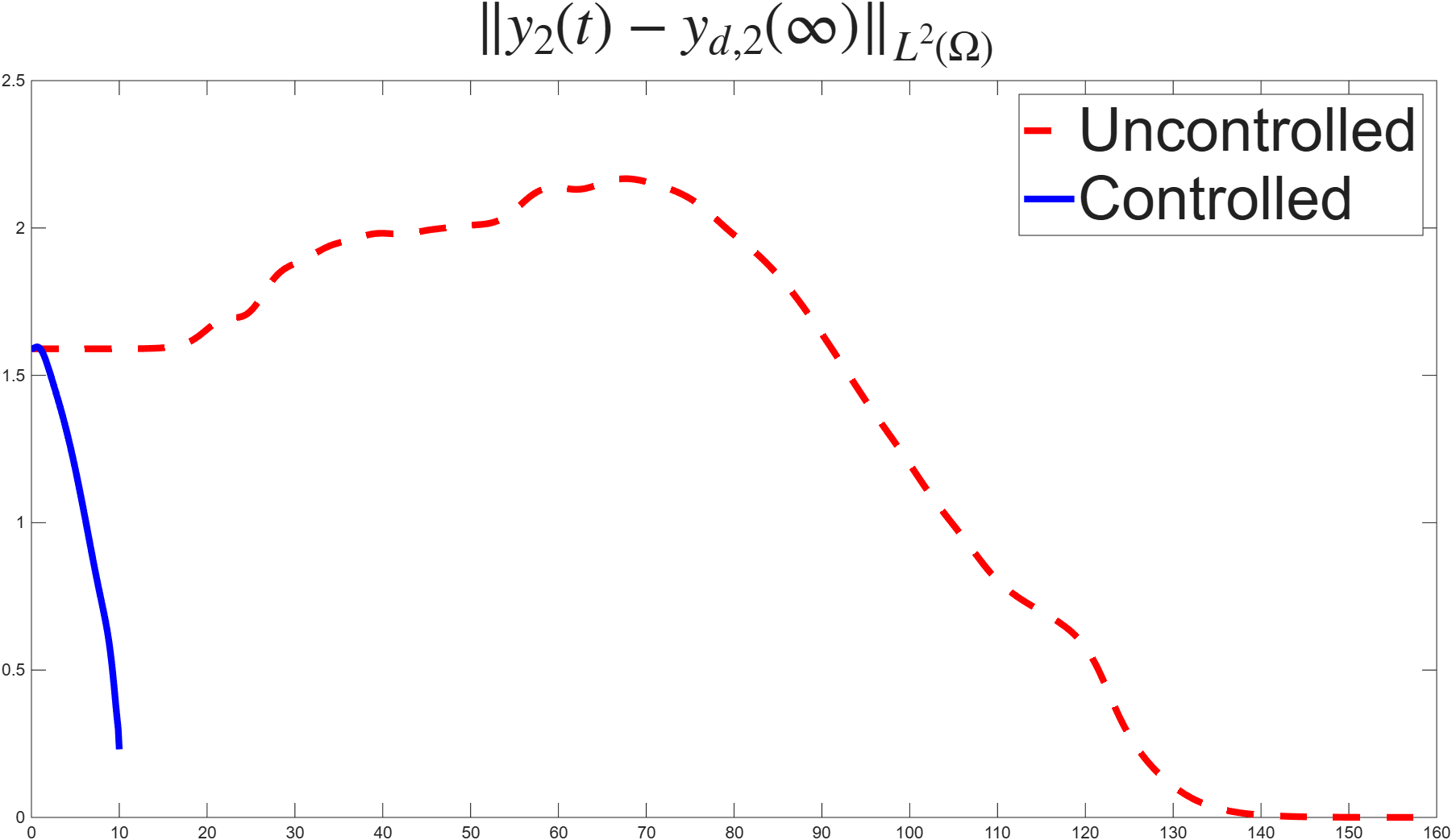}
  \caption{Example 2. }\label{F5}
\end{figure}

For $T=10$, we define the functional
\[J(u) = \frac{1}{2} \|y_{u,2}(T) - y_{d,2}(T_\infty)\|^2_{L^2(\Omega)} + \frac{1}{2} \|u\|^2_{L^2(\Sigma)},\]
and the parameters  $\umin = -1$, $\umax = 1$, $\ymin{1}=\ymin{2}=0.1$ and $\ymax{1}=\ymax{2}=+\infty$. We consider the following two control problems:
\[\mathrm{(\tilde P)}\min_{u\in\Uad} J(u)\qquad\text{ and } \qquad\mathrm{(P)}\min_{u\in\Aad} J(u).\]
We computed the solution $\tilde u$ of $\mathrm{(\tilde P)}$ and we found that its associated state $\tilde y$ takes negative values at some points; for instance, $\tilde y_2(0,1,4.625) =   -1.5406$. Such a state is not physically admissible, since \(y_1\) and \(y_2\) represent chemical concentrations. On the other hand, the solution of $\mathrm{(P)}$ yields a physically admissible state. Snapshots of $\bar y_2$ taken at regular intervals are shown in Figure \ref{F8}. In Figure \ref{F5} (blue line), we show the $L^2 (\Omega)$-distance between the target and the controlled activator at time $t$.

\begin{figure}[h!]
  \centering
  \includegraphics[width=.2\textwidth]{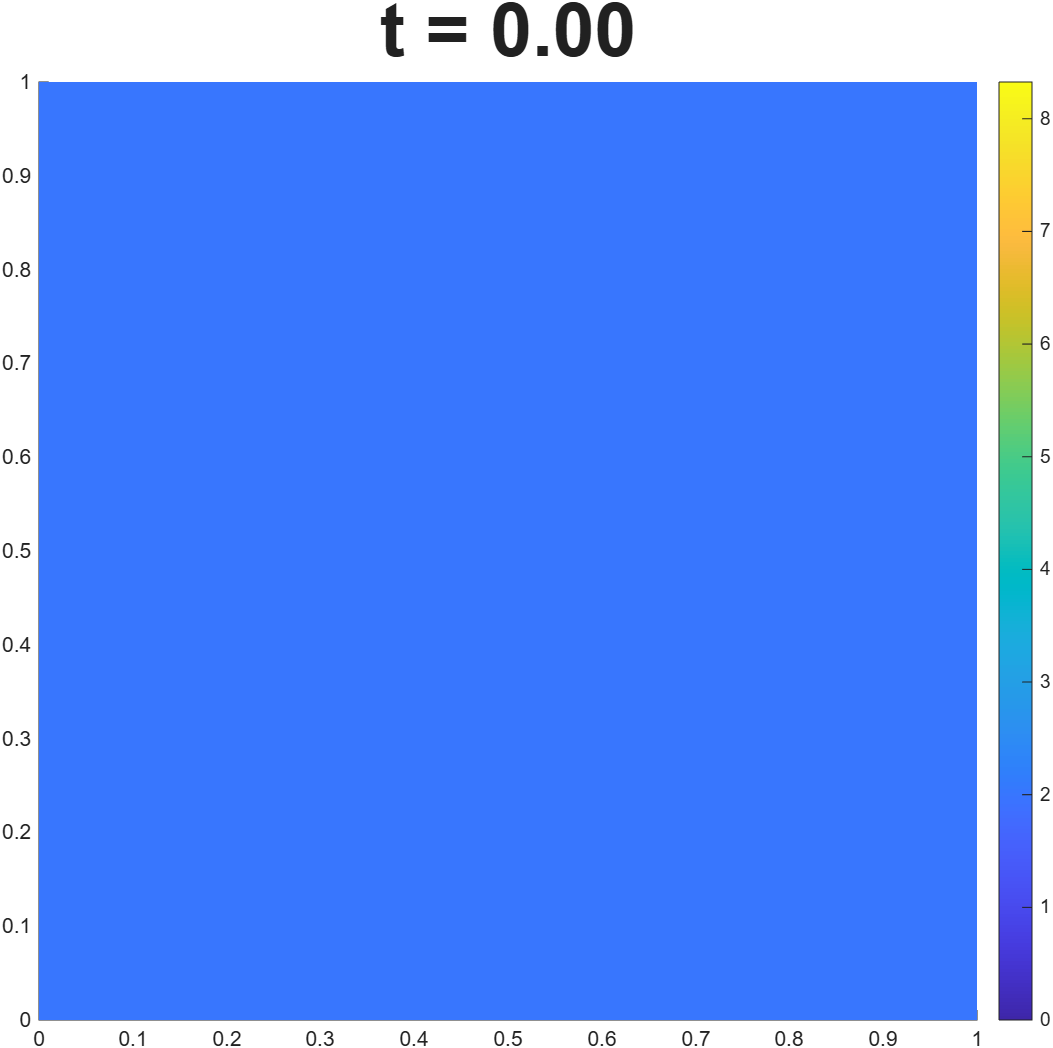}
  \includegraphics[width=.2\textwidth]{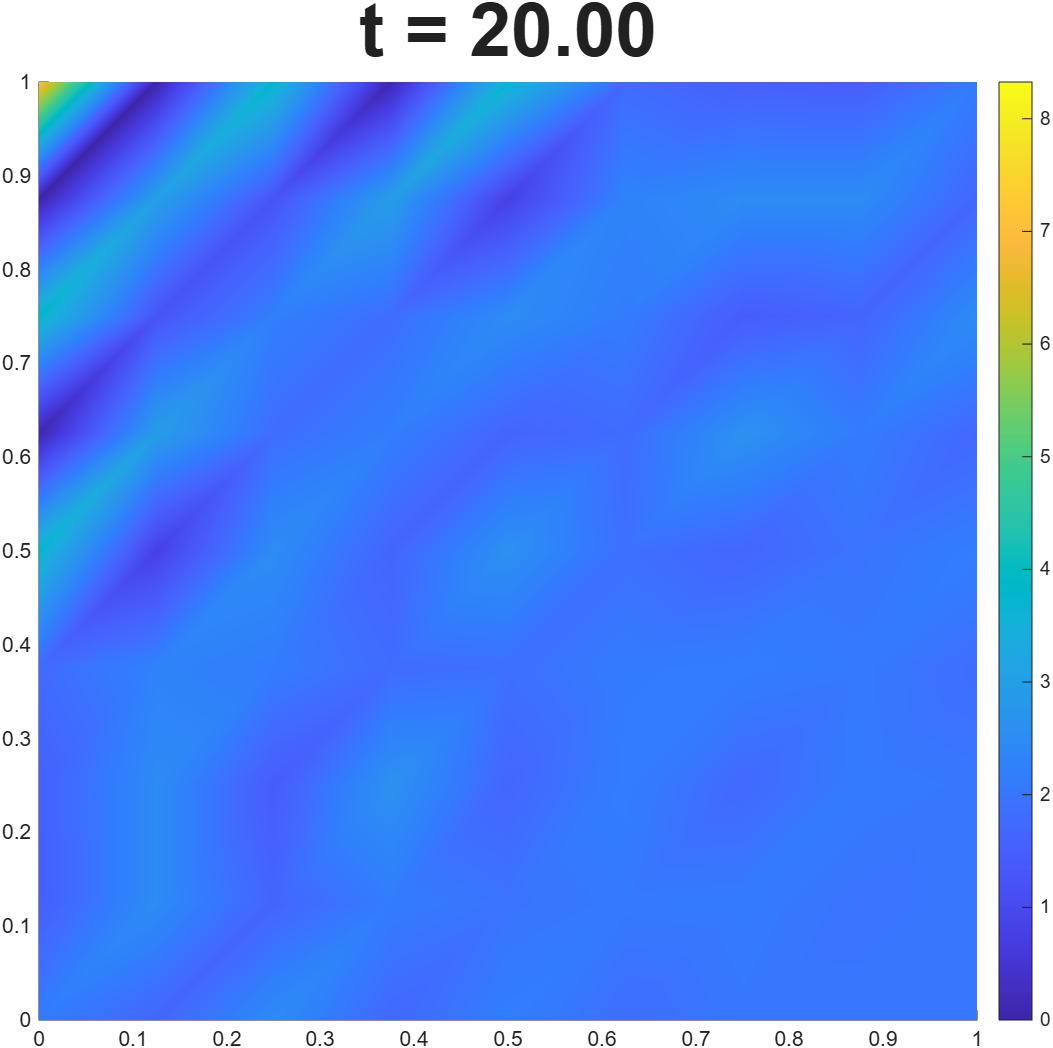}
  \includegraphics[width=.2\textwidth]{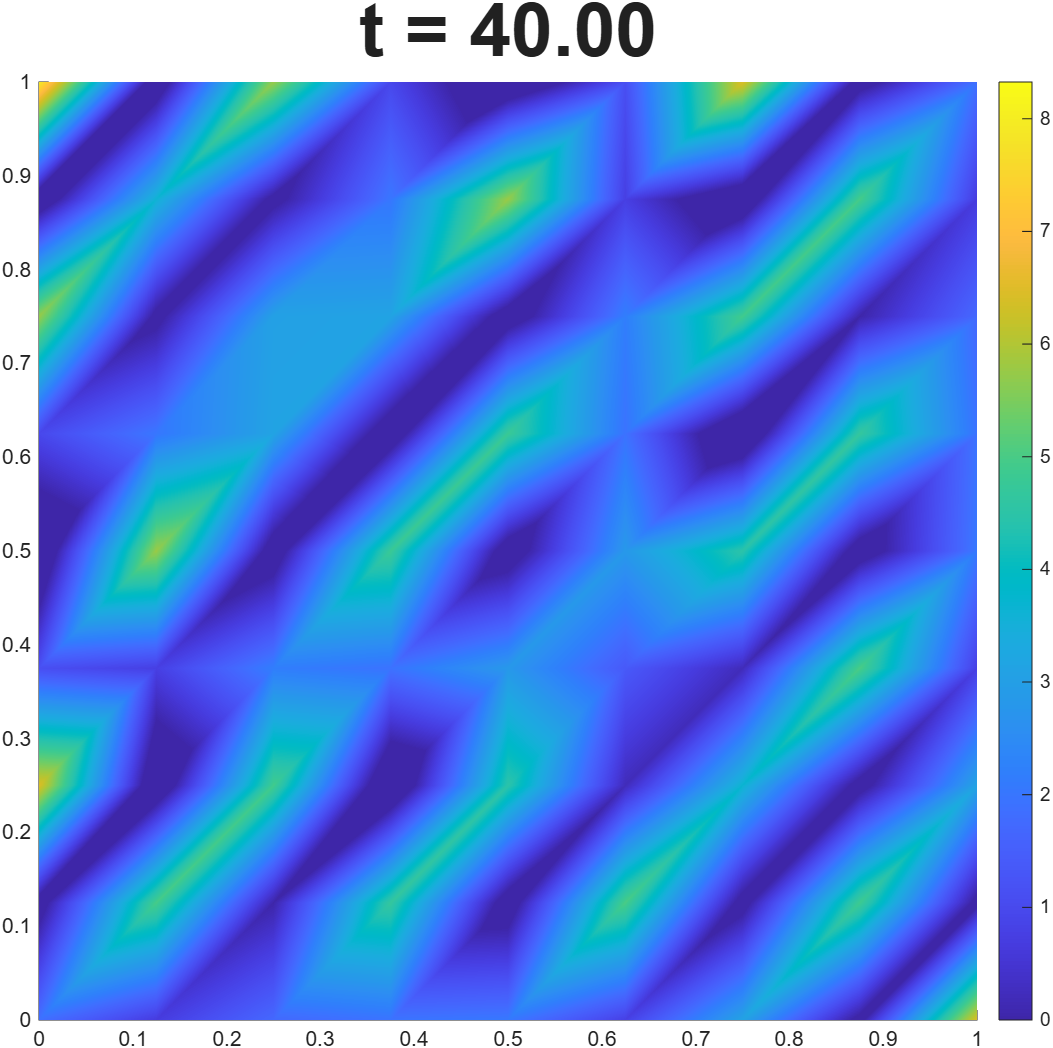}

  \includegraphics[width=.2\textwidth]{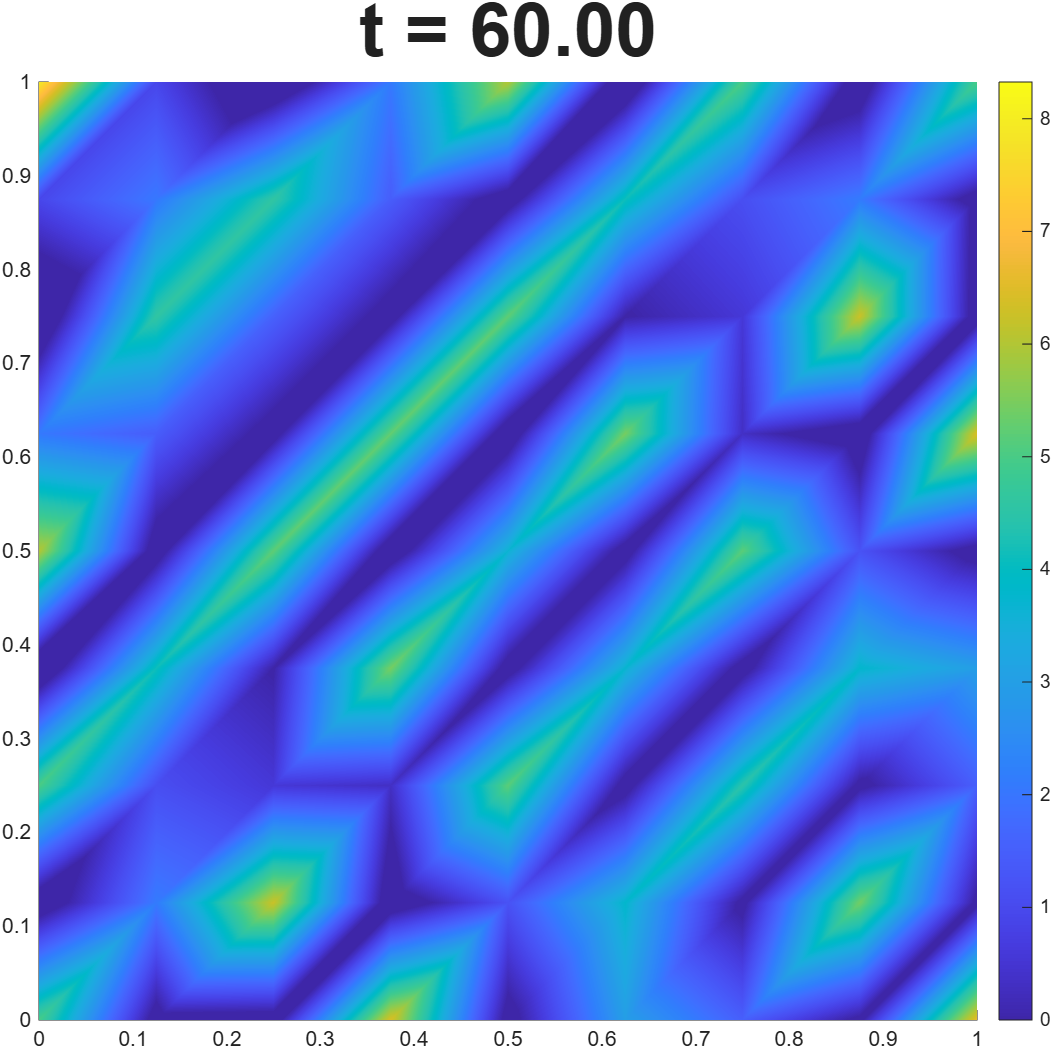}
  \includegraphics[width=.2\textwidth]{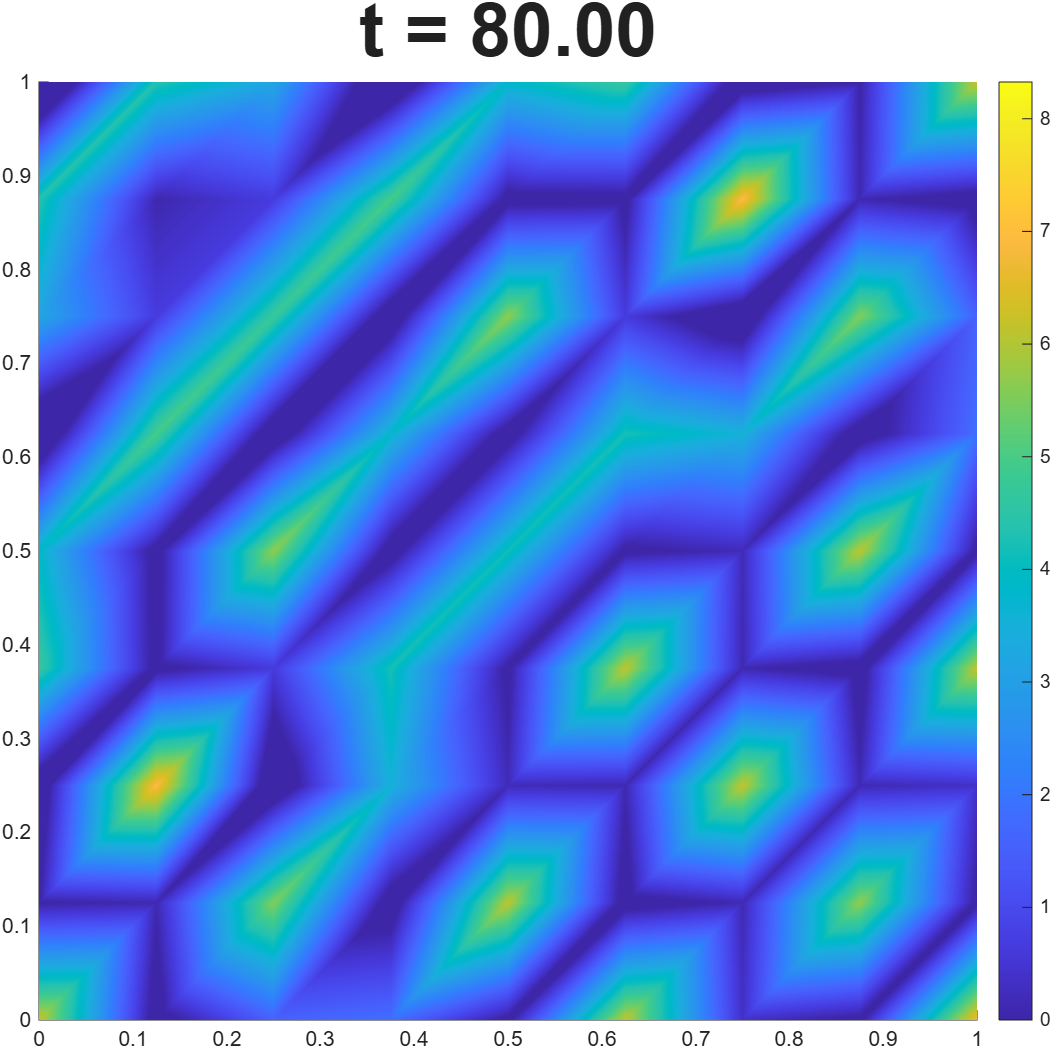}
  \includegraphics[width=.2\textwidth]{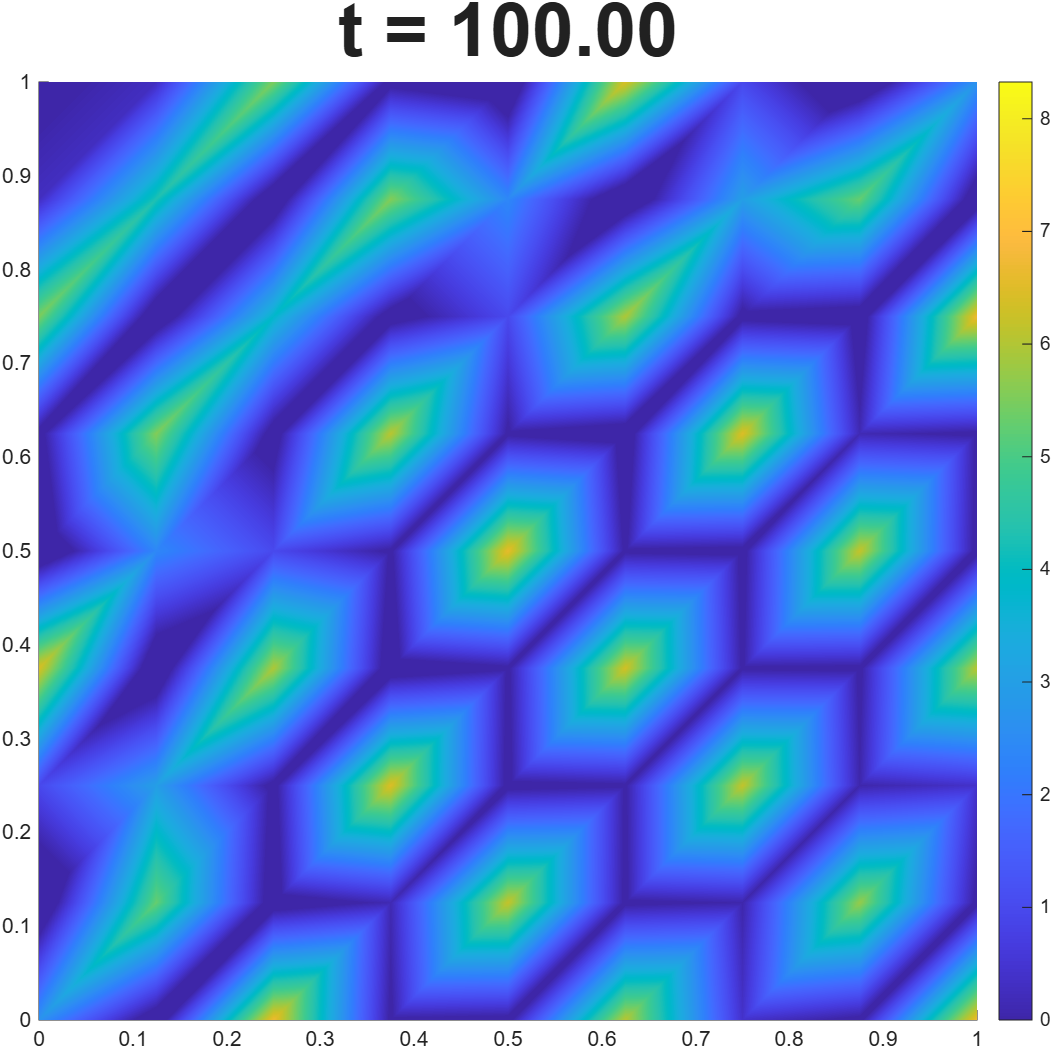}

  \includegraphics[width=.2\textwidth]{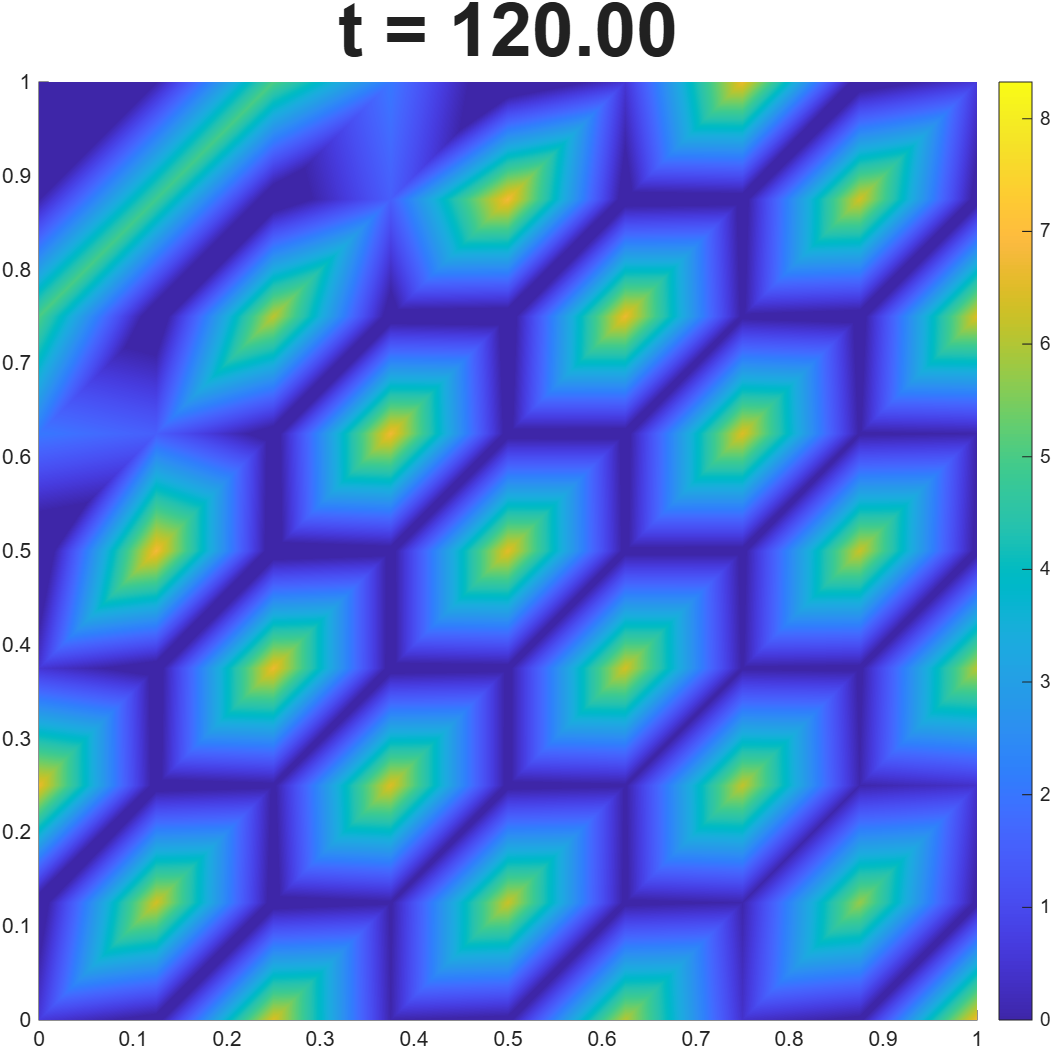}
  \includegraphics[width=.2\textwidth]{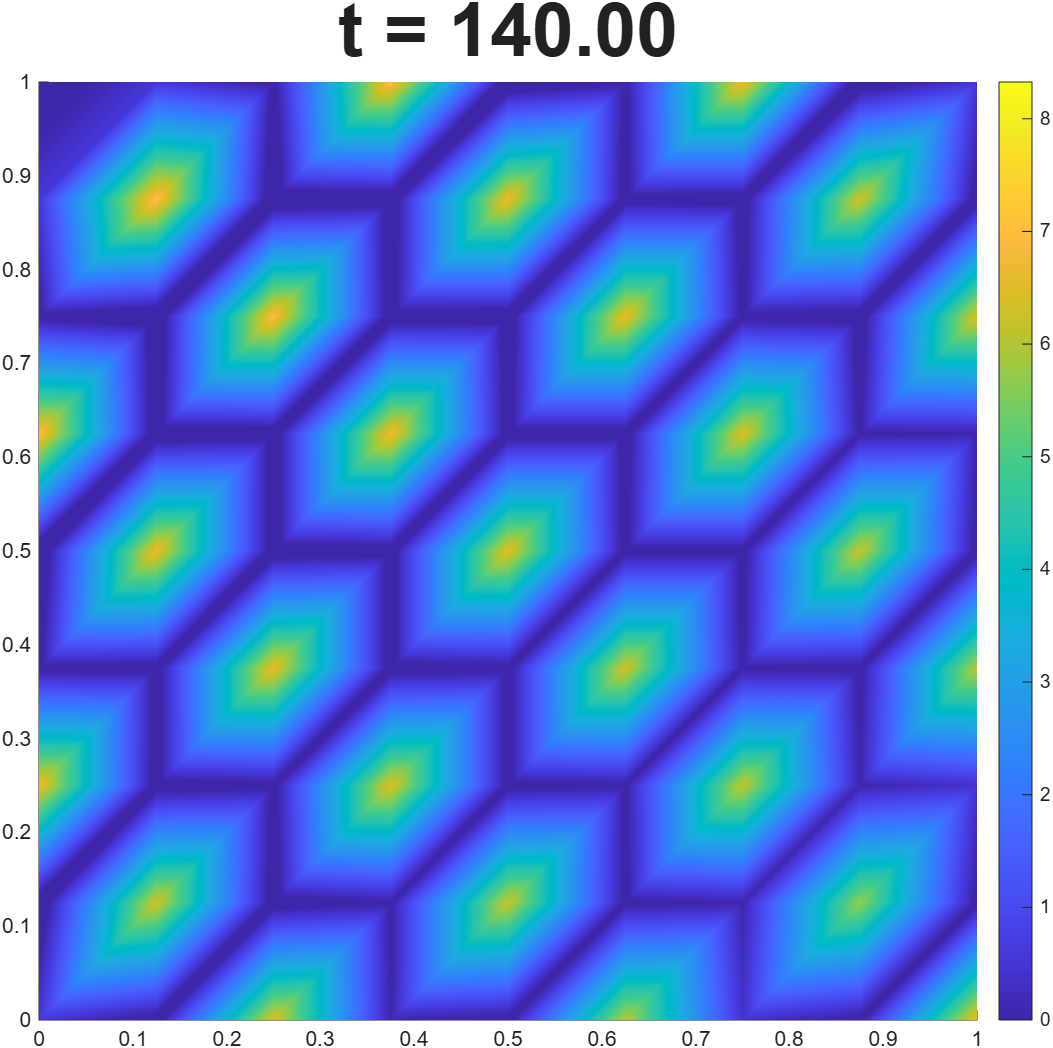}
  \includegraphics[width=.2\textwidth]{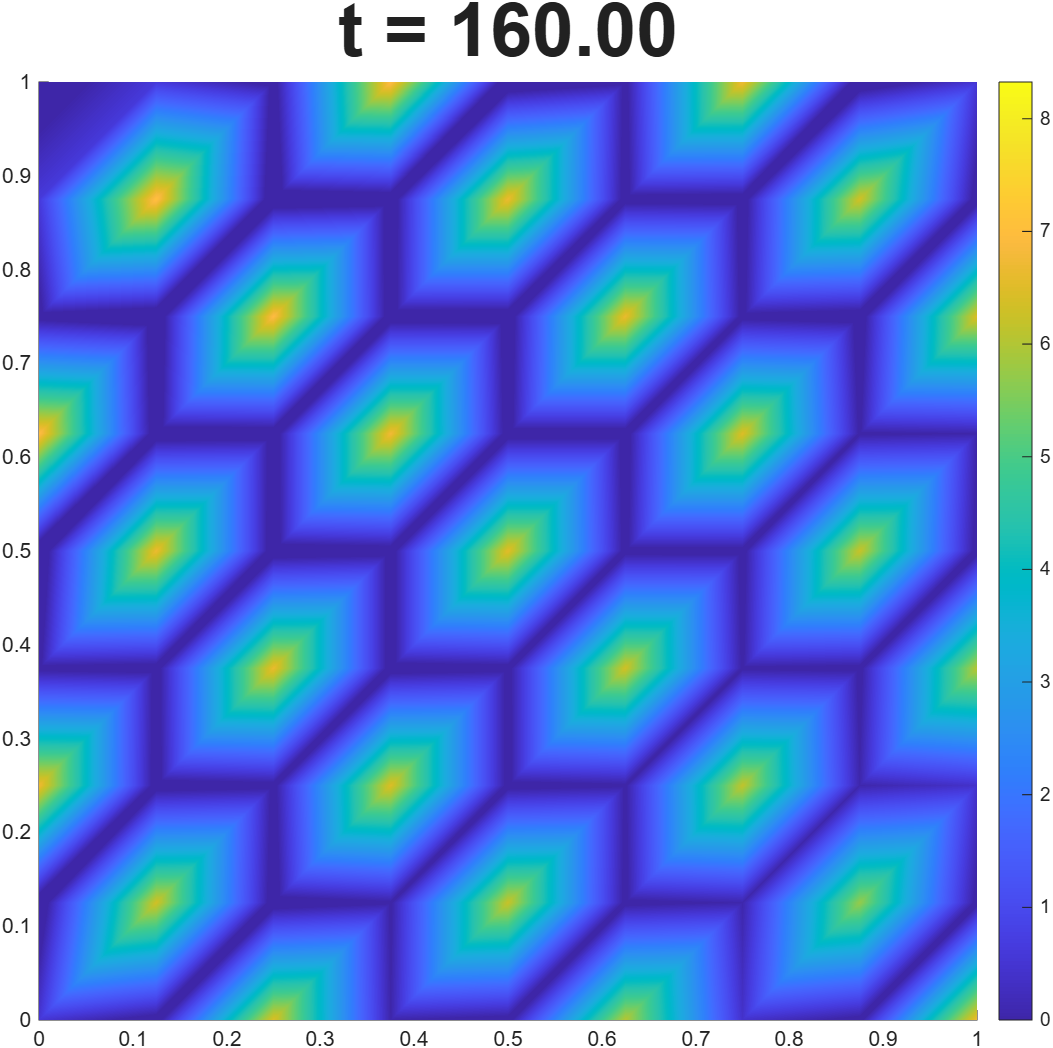} 
  \caption{Example 2. Evolution of the uncontrolled state $y_{d,2}$.}\label{F7}
\end{figure}

\begin{figure}[h!]
  \centering
  \includegraphics[width=.2\textwidth]{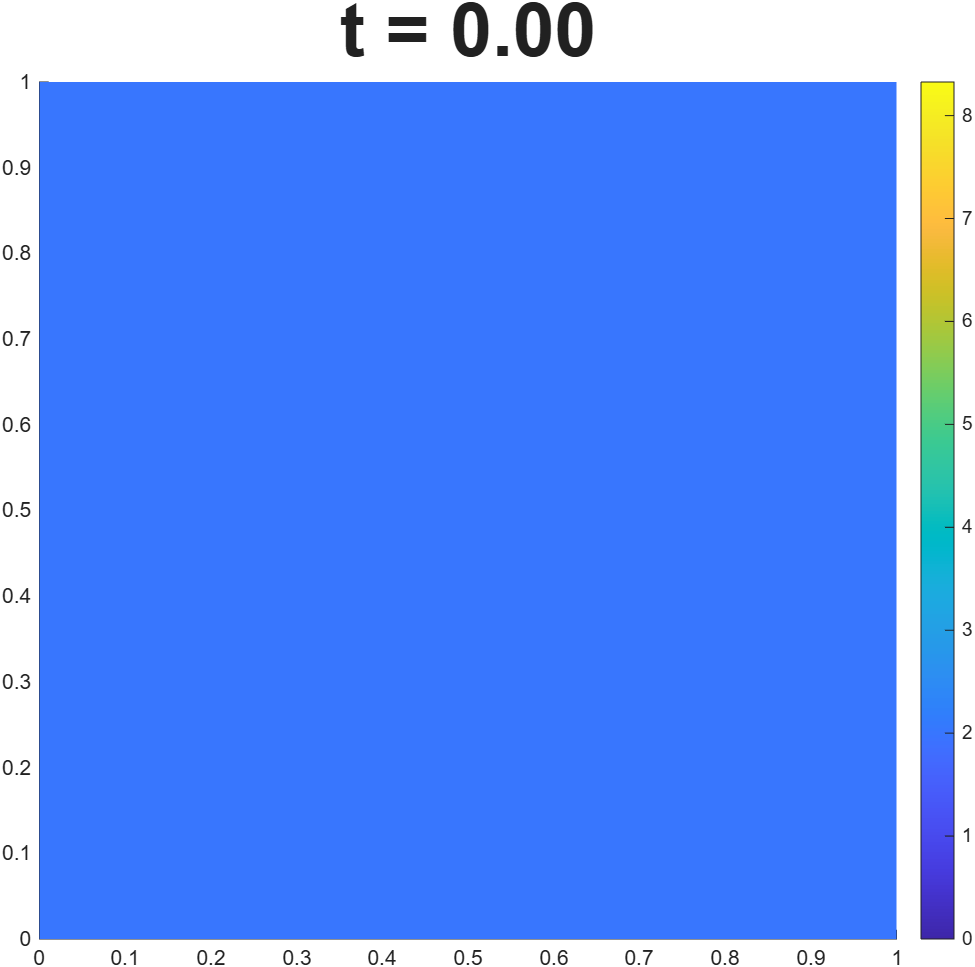}
  \includegraphics[width=.2\textwidth]{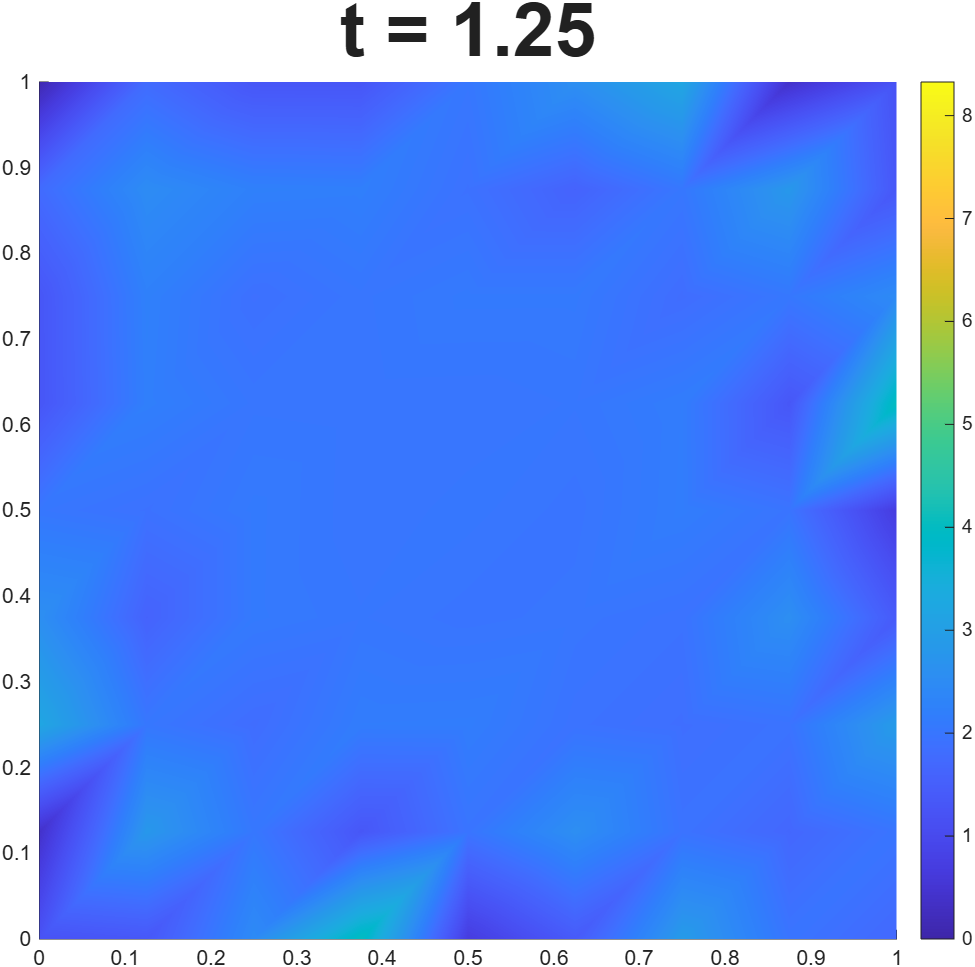}
  \includegraphics[width=.2\textwidth]{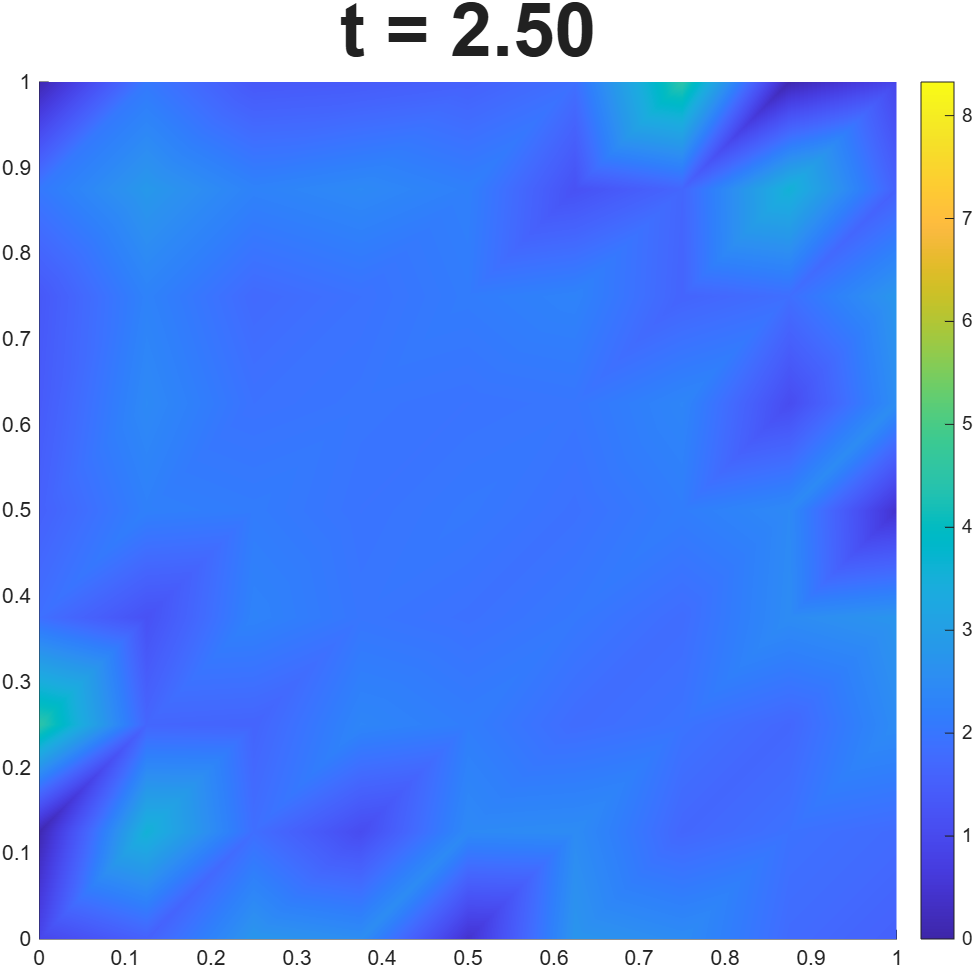}

  \includegraphics[width=.2\textwidth]{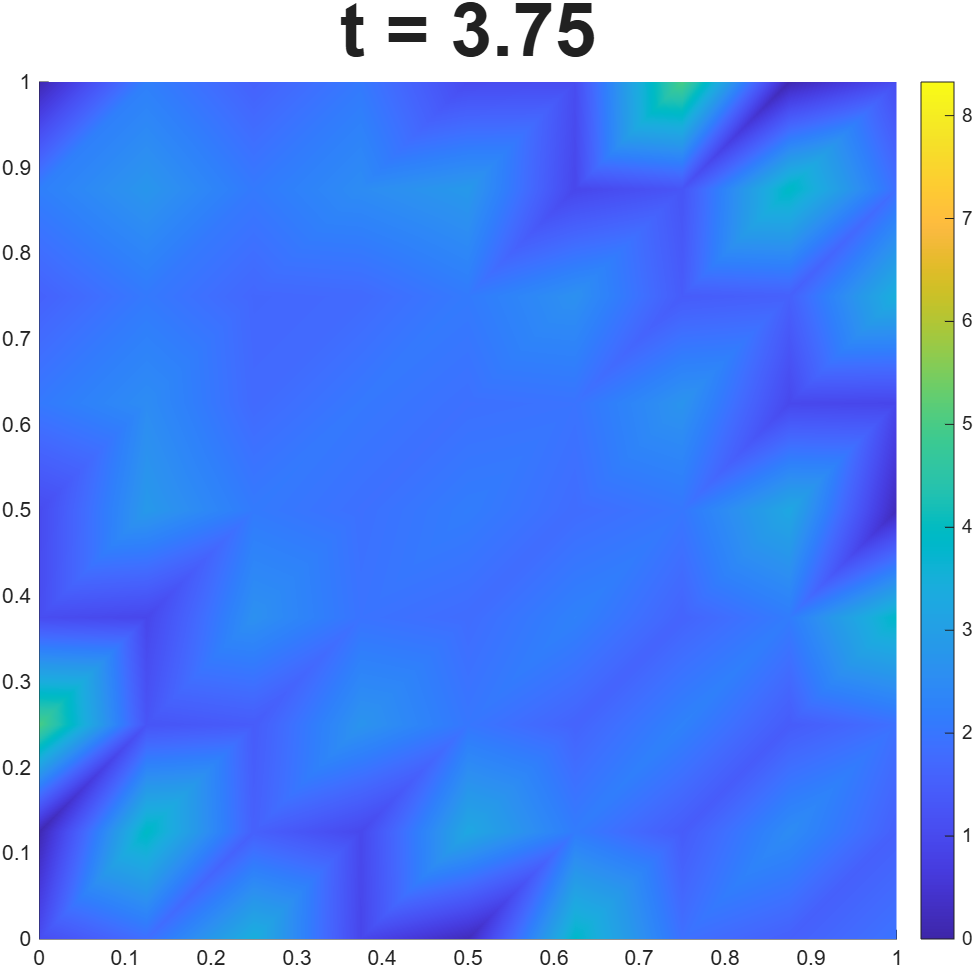}
  \includegraphics[width=.2\textwidth]{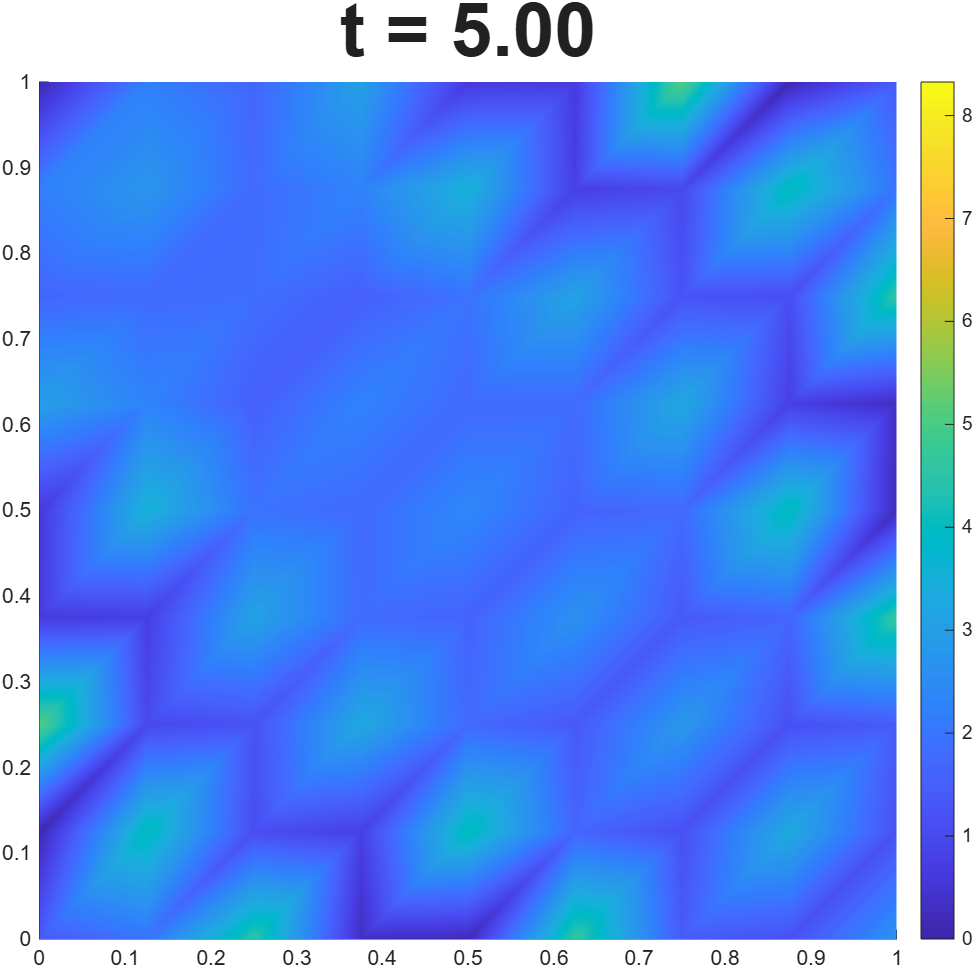}
  \includegraphics[width=.2\textwidth]{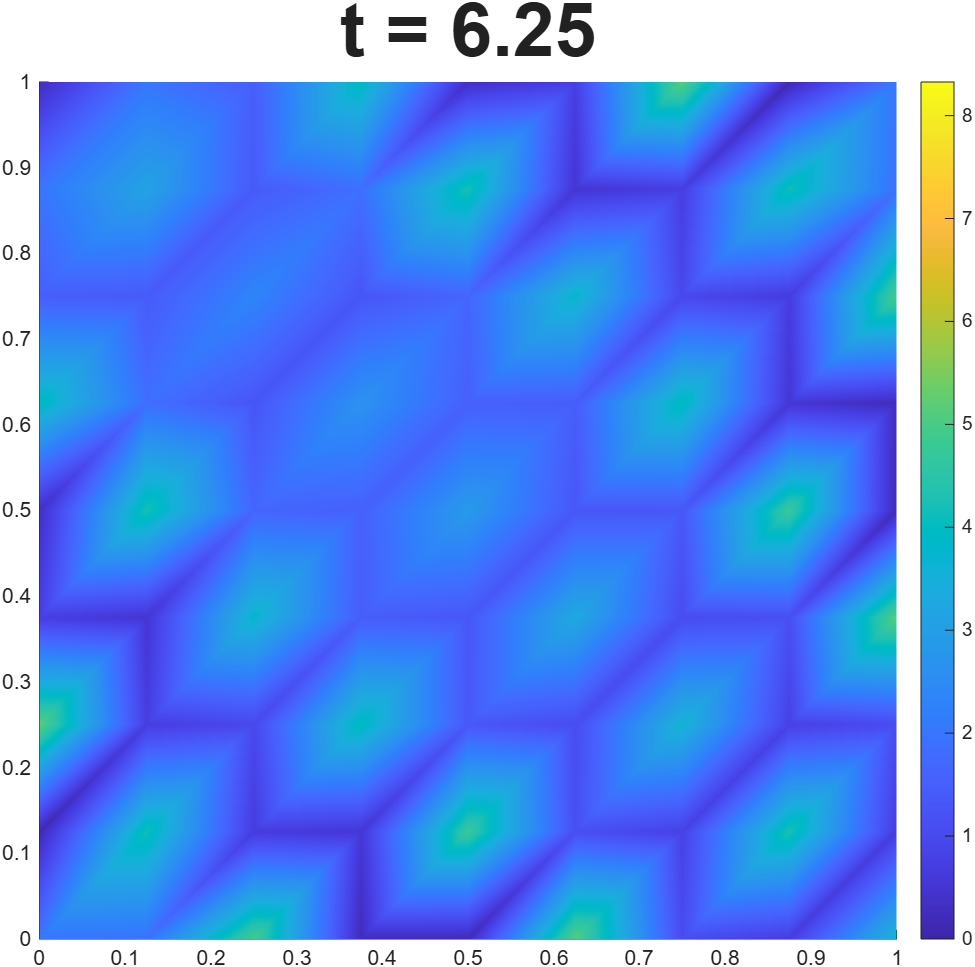}

  \includegraphics[width=.2\textwidth]{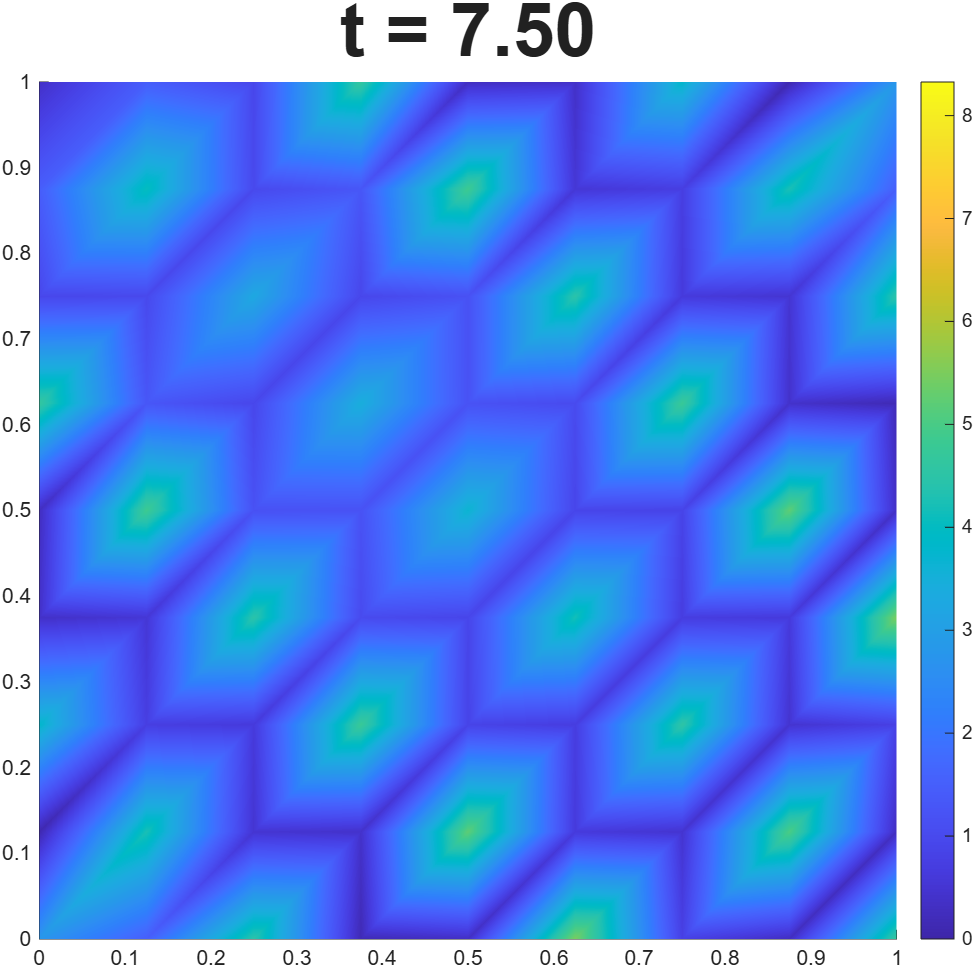}
  \includegraphics[width=.2\textwidth]{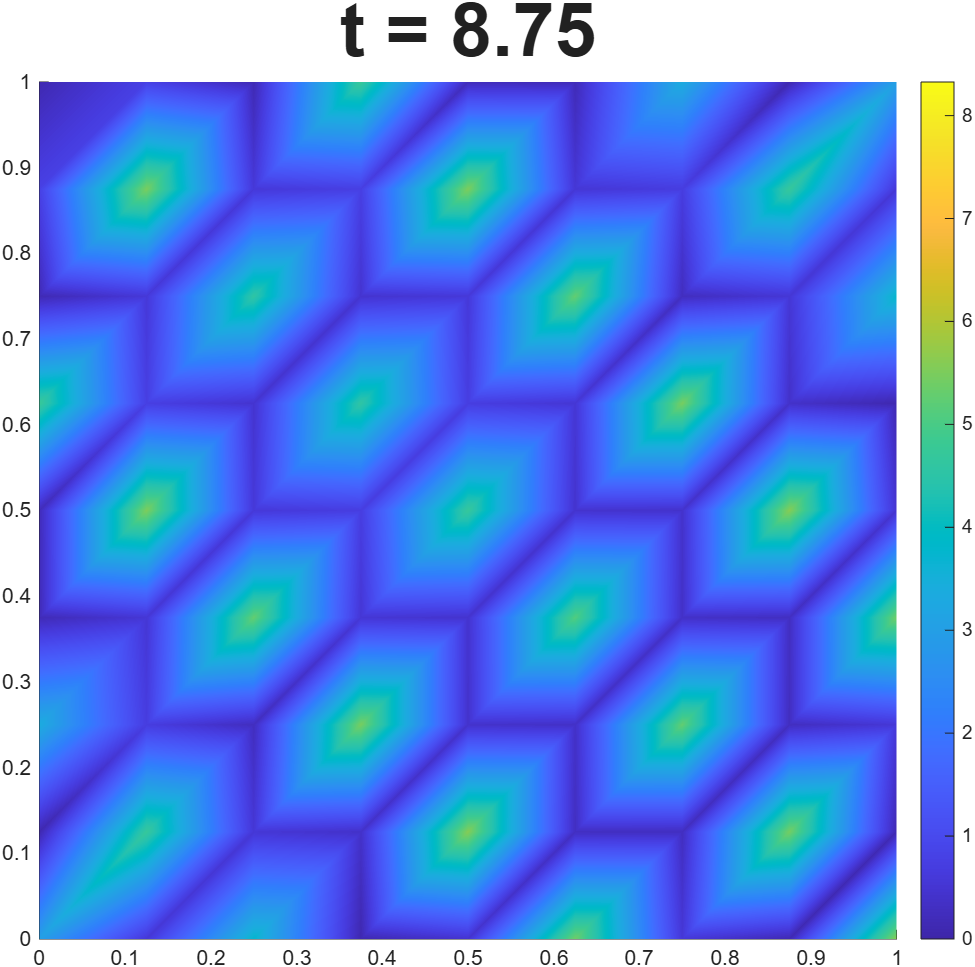}
  \includegraphics[width=.2\textwidth]{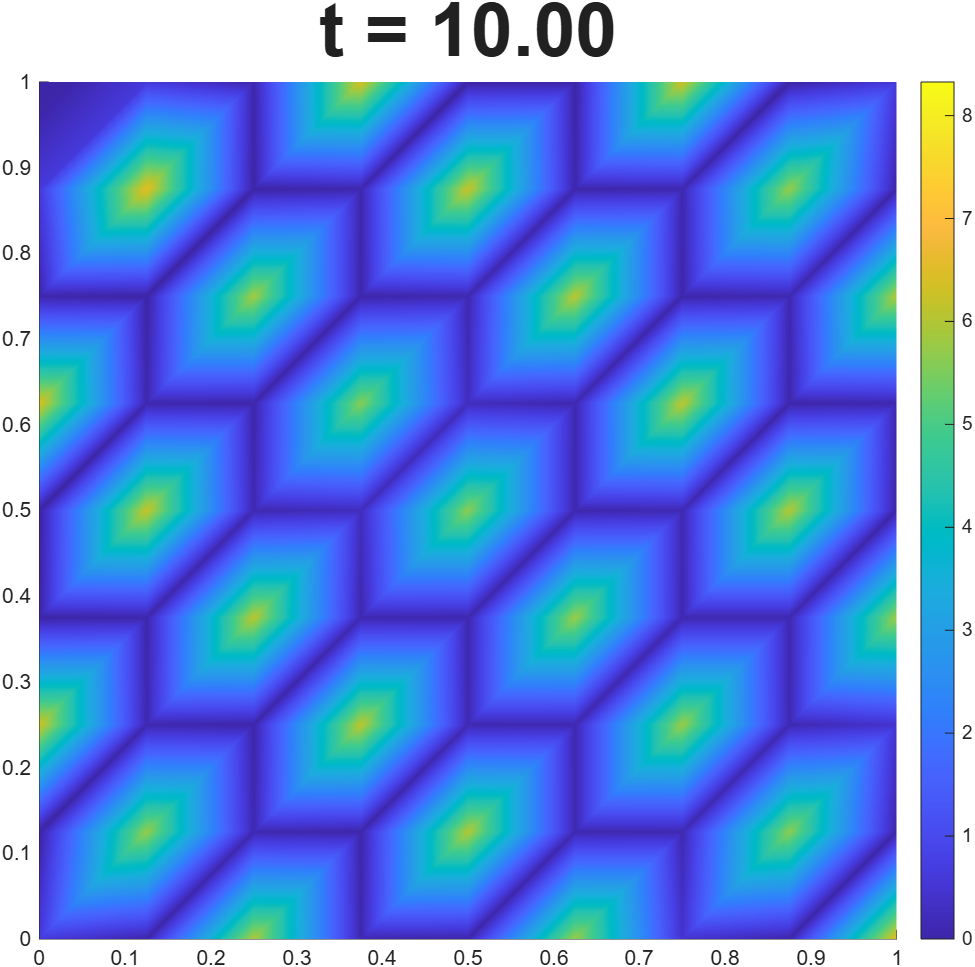}
  \caption{Example 2. Evolution of the controlled state $\bar y_2$.}\label{F8}
\end{figure}

\textit{Some computational details.} To obtain a solution to the state-constrained problem, we have used the alternating direction method of multipliers (ADMM) presented in \cite{Glowinskietal2020}. At each iteration, we solve the control-constrained problem using a trust-region method based on SQP applied to the reduced functional. This means that the nonlinear state equation must be solved at every iteration; see \cite{CasasMateos2026a}. We use the quasi-Newton L-BFGS method during the first iterations and switch to Newton's method once a sufficiently accurate approximation of the solution has been obtained. The control-constrained quadratic program arising at each SQP step is solved with a semismooth Newton method (although the control constraints are not active in our examples). The unconstrained quadratic program that arises at each step of the semismooth Newton method is solved using the conjugate gradient method. The initial iterate for the ADMM is the solution $\tilde u$ of the problem without state constraints.

We solve the partial differential equations using the finite element method in space with a uniform spatial triangular mesh of size $h$, while the Crank-Nicolson method is used in time with a uniform time partition of size $\tau$. All four variables---state, adjoint state, control, and Lagrange multiplier---are discretized in space with continuous piecewise linear Lagrange finite elements. To handle the nonlinear terms, a seven-point Gaussian formula of fifth degree has been used for numerical integration, so that all spatial integrals are evaluated exactly. The Tikhonov regularization term is discretized using the lumped mass matrix method in space and the midpoint rule in time. The state is discretized in time using continuous piecewise linear elements. In contrast to \cite{BeckerMeidnerVexler2007}, the time discretization of the adjoint state is performed also using continuous piecewise linear elements. A suitable use of the midpoint rule, cf. \cite[eq. (1.54) and (13.20)]{Thomee2006}, for the time integrals leads to a more natural backward-in-time scheme for solving the discrete adjoint equation. This avoids the half-steps required in \cite[Remark 3.2]{BeckerMeidnerVexler2007}, while still yielding the correct discrete gradient. This technique also allows both the control and the Lagrange multiplier to be discretized in time using piecewise constant functions. All the software was developed in-house using \textsc{Matlab} R2026a.


\end{document}